\documentclass[a4paper,10pt,reqno]{amsart} %
\usepackage{latexsym,amssymb}
\usepackage{hyperref}
\usepackage{amsmath,amsthm,amsxtra}
\usepackage{amsfonts}
\usepackage[american]{babel}
\usepackage{mathrsfs}
\usepackage{psfrag}
\usepackage{epsfig,inputenc}
\usepackage{graphpap,latexsym,epsf}
\usepackage{color}
\usepackage{amssymb,eucal,paralist,color,enumerate}
\usepackage{graphicx}

\newcommand{\ep}{\varepsilon}
\newcommand{\R}{\mathbb{R}}
\newcommand{\N}{\mathbb{N}}

\newcommand{\pa}{\partial}

\mathchardef\emptyset="001F

\newtheorem{maintheorem}{Theorem}
\newtheorem{maindefinition}{Definition}
\newtheorem{theorem}{Theorem}[section]
\newtheorem{prop}[theorem]{Proposition}

\newtheorem{lemma}[theorem]{Lemma}
\newtheorem{remark}[theorem]{Remark}

\newtheorem{definition}[theorem]{Definition}
\newtheorem{proposition}[theorem]{Proposition}

\newtheorem{notation}[theorem]{Notation}

\numberwithin{equation}{section}

\newcommand{\eps}{\varepsilon}

\newcommand{\down}{\downarrow}
\newcommand{\weaksto}{\rightharpoonup^*}

\newcommand{\AC}{\mathrm{AC}}

\newcommand{\Hilbert}{\xfin}
\newcommand{\dualoperator}

  \def\calC{{\mathcal C}}
 \def\calE{{\mathcal E}} 
\def\calG{{\mathcal G}}

  \def\rmC{{\mathrm C}}
\def\rmD{{\mathrm D}}

\def\dd{\;\!\mathrm{d}} 

\newcommand{\pairing}[4]{ \sideset{_{ #1 }}{_{ #2 }}  {\mathop{\langle #3 , #4
\rangle}}}

\newcommand{\teta}{\vartheta}

\newcommand{\nchi}{{\raise.2ex\hbox{$\chi$}}}

\definecolor{ddcyan}{rgb}{0,0.1,0.9}
\definecolor{ddmagenta}{rgb}{0.8,0,0.8}
\definecolor{orange}{rgb}{0.6,0.2,0}

\newcommand{\piecewiseConstant}[2]{\overline{#1}_{\kern-1pt#2}}

\newcommand{\foraa}{\text{for a.a. }}

\newcommand{\cE}{\mathcal{E}}
\newcommand{\cg}[1]{\mathcal{G}(#1)}
\newcommand{\ene}[2]{\mathcal{E}_{#1}(#2)}
\newcommand{\subl}[1]{\mathcal{S}_{#1}}
\newcommand{\pt}[2]{\mathcal{P}_{#1}(#2)}
\newcommand{\Ptname}{\mathcal{P}}
\newcommand{\admis}[3]{\mathscr{A}_{#1,#2}^{#3}}
\newcommand{\domainenergy}{\xfin}
\newcommand{\la}{\langle}
\newcommand{\ra}{\rangle}
\newcommand{\cost}[3]{\mathsf{c}_{#1}(#2;#3)}
\newcommand{\costname}[1]{\mathsf{c}_{#1}}

\newcommand{\limen}{\mathscr{E}}
\newcommand{\limp}{\mathscr{P}}
\newcommand{\BV}{\mathrm{BV}}

 \def\trait #1 #2 #3 {\vrule width #1pt height #2pt depth #3pt}

 \def\fin{\hfill
         \trait .3 5 0
         \trait 5 .3 0
         \kern-5pt
         \trait 5 5 -4.7
         \trait 0.3 5 0
 \medskip}
 
\newcommand{\QED}{\mbox{}\hfill\rule{5pt}{5pt}\medskip\par}

\newcommand{\mdt}[3]{\|{#1}'_{#2}(#3)\|}
\newcommand{\mdtq}[4]{\|{#1}'_{#2}(#3)\|^{#4}}

\newcommand{\minpartial}[3]{\|\rmD {#1}_{#2}(#3)\|}

\newcommand{\minpartialq}[3]{\|\rmD {#1}_{#2}(#3)\|^2}
\newcommand{\argminpartial}[3]{\rmD{#1}_{#2}(#3)}

\newcommand{\pij}[2]{\pi_{#2}(#1)}

\newcommand{\gder}[4]{\mathrm{D}^{#1} {#2}_{#3}(#4)}

\newcommand{\NN}{\mathrm{N}}

\newcommand{\xfin}{X}

\newenvironment{rcomm}{\color{red}}{\color{black}}

\newenvironment{rnew}{\color{ddmagenta}}{\color{black}}

\newcommand{\berin}{\begin{rnew}}
\newcommand{\erin}{\end{rnew}}

\newcommand{\beroc}{\begin{rcomm}}
\newcommand{\eroc}{\end{rcomm}}

\newenvironment{newricky}{\color{ddcyan}}{\color{black}}
\newcommand{\bnr}{\begin{newricky}}
\newcommand{\enr}{\end{newricky}}

\definecolor{vgreen}{rgb}{0.1,0.5,0.2}

\title[Singularly perturbed gradient flows]{Singular vanishing-viscosity limits of gradient flows:\\ 
the finite-dimensional case}
\author{Virginia Agostiniani}
\address{V.\ Agostiniani, SISSA, via Bonomea 265, 34136 Trieste - Italy} 
\email{vagostin\,@\,sissa.it}

\author{Riccarda Rossi}
\address{R.\ Rossi, DIMI, Universit\`a degli studi di Brescia, via Branze 38, 25133 Brescia - Italy}
\email{riccarda.rossi\,@\,unibs.it}

\thanks{V.A.\ has been partially funded by the 
European Research Council / ERC Advanced Grant ${n^o}\ 340685$.
R.R.\ has been partially supported by a MIUR-PRIN'10-11 grant
for the project ``Calculus of Variations''. V.A.\ and  R.R.\  also acknowledge
support from the  Gruppo Nazionale per  l'Analisi Matematica, la
Probabilit\`a  e le loro Applicazioni (GNAMPA) of the Istituto Nazionale di Alta Matematica (INdAM)}

\date{November 23, 2016}
\begin{document}

\begin{abstract}
In this note
we study the singular vanishing-viscosity limit of a gradient flow set in a finite-dimensional Hilbert space and driven by a smooth, but possibly \emph{nonconvex},
time-dependent energy functional. We resort to  ideas and techniques from the variational approach to gradient flows and rate-independent evolution to show that, under suitable assumptions,
the solutions to the singularly perturbed problem converge to a curve of stationary points of the energy, whose   behavior at jump points is characterized in terms of the notion of \emph{Dissipative Viscosity} solution. We also provide sufficient conditions under which  Dissipative Viscosity solutions enjoy better properties, which  turn them  into   \emph{Balanced Viscosity} solutions. Finally, we discuss the \emph{generic character} of our assumptions. 
\end{abstract} 

\maketitle


\section{Introduction}

 We   address the singular  limit, 
as $\eps \downarrow 0$, of the gradient flow equation
\begin{equation}
\label{e:sing-perturb}
\ep u'(t)+\rmD\mathcal E_t(u(t))
= 0 
\qquad \text{in $\xfin$}  \quad \foraa\, 
t \in [0,T].
\end{equation}
Here, $(\xfin, \| \cdot\|)$ is a finite-dimensional Hilbert space,  
and the driving energy functional $\calE$ is \emph{smooth}, i.e. 
\begin{equation}
\label{Ezero}
\tag{$\mathrm{E}_0$}
\calE \in \rmC^1([0,T]\times \xfin)
\end{equation}
($\rmD\calE$ denoting the differential with respect to the variable $u$), but we allow for the  mapping  $u\mapsto \ene tu$ to be \emph{nonconvex}.
In this paper we aim to enucleate the basic ideas underlying a 
\emph{novel, variational} approach to this singular perturbation problem, 
partially inspired by the theory of \emph{Balanced Viscosity solutions} 
to rate-independent systems  \cite{EfMi06, MiRoSa09, MiRoSa, MRS13_prep}. 
Let us mention that this approach 
can be in fact adapted, and refined,  
to study the  singular limit \eqref{e:sing-perturb} 
in an  infinite-dimensional Hilbertian setting,  and with a possibly \emph{nonsmooth}, as well as \emph{nonconvex},  
driving energy functional  $\calE$, cf.\ the forthcoming \cite{lavorone}. 
The simpler setting considered in this paper enables us to illustrate the 
cornerstones of our analysis, unhampered by the technical issues related to nonsmoothness and infinite dimensionality. 
We will prove  the convergence as $\eps\down0$ of  (sequences of) solutions to (the Cauchy problem for) \eqref{e:sing-perturb},
to a curve $u:[0,T]\to \xfin$ of critical points for $\calE$, i.e.\ fulfilling the stationary problem
\begin{equation}
\label{e:lim-eq}
\rmD \ene t{u(t)} = 0 \quad \text{in } \xfin  \quad \foraa\, t \in[0,T].
\end{equation}
 The properties of $u$  will be codified by the two different notions of \emph{Dissipative Viscosity} and \emph{Balanced Viscosity} solution. 
 \par
 Before illustrating our results, let us  hint at the main analytical difficulties attached to the asymptotic analysis of \eqref{e:sing-perturb} as $\eps \down 0$, as well as  at the results available in the literature.  
In particular,  in the following lines we will focus on the case of    \emph{(uniformly) convex} energies, and of   
 energy functionals $\ene t{\cdot}$ complying with the \emph{transversality conditions}.
 Let us also briefly mention that new  results have emerged  in the recent \cite{Artina-et-al} for \emph{linearly constrained} evolution of critical points, based on a  constructive approach instead of the vanishing-viscosity analysis of \eqref{e:sing-perturb}. 
  
 \paragraph{\bf Preliminary considerations.}
Under suitable conditions,   for every fixed $\eps>0$ and
   for every $u_0 \in \xfin$  there exists at least a solution
  $u_\eps\in H^1 (0,T;\xfin)$ to the gradient flow \eqref{e:sing-perturb}, fulfilling
  the Cauchy condition $u_\eps(0)=u_0$.
   Testing \eqref{e:sing-perturb} by $u_\eps'$, integrating in time, and exploiting the chain rule for
  $\mathcal{E}$, 
  it is immediate to check that
  $u_\eps$ complies with the \emph{energy identity}
\begin{equation}
\label{enid-intro}
\int_s^t \eps \|u_\eps'(r)\|^2 \dd r  + \mathcal{E}_t (u_\eps(t)) = \mathcal{E}_s(u_\eps(s))+
\int_s^t \partial_t
\mathcal{E}_r(u_\eps(r))   \dd  r \quad \text{for all $ 0 \leq s \leq t \leq T$},
\end{equation}
balancing the dissipated energy $\int_s^t \eps \|u_\eps'(r)\|^2 \dd r $ with the stored energy and with  the work of the external forces $\int_s^t \partial_t
\mathcal{E}_r(u_\eps(r))   \dd  r$.
From  \eqref{enid-intro}
all the a priori estimates on a family $(u_\eps)_{\eps}$ of solutions can be deduced. 
More specifically, using the \emph{power control} condition $|\partial_t \ene tu|\leq C_1 
\ene tu +C_2$ for some $C_1,\, C_2>0$, 
via the Gronwall Lemma
one obtains
\begin{equation}
\label{prelim-estimates-intro}
\begin{aligned}
&
\text{\emph{(i)}}  && \text{
 The energy bound $\sup_{t\in (0,T)} \ene t{u_\eps(t)} \leq C $};
 \\
 &
\text{\emph{(ii)}}  && \text{
 The estimate $\int_0^T \eps \|u_\eps'(t) \|^2 \dd t\leq C'$,} 
\end{aligned}
\end{equation}
for positive constants $C, \, C'>0$ independent of $\eps>0$. 
While \emph{(i)}, joint with a suitable \emph{coercivity} condition on $\calE$ (typically, compactness of the energy sublevels), 
  yields   that there exists a \emph{compact} set  $K\subset\xfin$  s.t.\ 
$u_\eps(t) \in K$ for all $t\in [0,T]$ and $\eps>0$,  the \emph{equicontinuity} estimate provided
  by \emph{(ii)} \emph{degenerates} as $\eps\down 0$. Thus, no Arzel\`a-Ascoli type result applies to deduce compactness for $(u_\eps)_\eps$. 
 This is the  major difficulty 
  in the asymptotic analysis of \eqref{e:sing-perturb}.
   \par
   Let us point out that this obstruction  can be circumvented  by \emph{convexity} arguments.
  Indeed,  if 
  $\mathcal{E} \in \mathrm{C}^2 ([0,T]\times \xfin)$ with the mapping $u \mapsto \ene t{u}$ uniformly convex, 
  then, starting from any  
   $u_0 \in \xfin $ with $\mathrm{D} \mathcal{E}_0(u_0)=0$  and $\mathrm{D}^2  \mathcal{E}_0(u_0)$ 
(with $\mathrm{D}^2 \calE  $ the second order derivative of $\calE$ w.r.t.\ $u$) positive definite 
(then $u_0$ is a \emph{non-degenerate} 
critical point of $\ene 0{\cdot}$), it can be shown
there exists a \emph{unique} curve $u \in \mathrm{C}^1([0,T];\xfin)$ of stationary points,
   to which the \emph{whole} family $(u_\eps)_\eps$ converge as $\eps\down 0$, uniformly on $[0,T]$.
\par
  Therefore, it is indeed significant to focus  on the case in which
the energy  $u \mapsto\ene ty$   is allowed to be
\emph{nonconvex}.  In this context, two problems arise:
\begin{enumerate}
\item Prove that, up to the extraction of a subsequence,
the gradient flows $(u_\eps)_\eps$ converge 
as $\eps \down 0$ to some limit
curve $u$, pointwise in $[0,T]$;
\item
 Describe the evolution of $u$. Namely, one expects $u$
to be a curve of critical points, jumping 
at  \emph{degenerate}  critical points  for
$\mathcal{E}_t(\cdot)$. In this connection, one  aims to provide a thorough  description of the energetic behavior of $u$ at jump points.
\end{enumerate}
\paragraph{\bf Results for \emph{smooth} energies in finite dimension: the approach via the \emph{transversality conditions}.}
For the singular perturbation limit
 \eqref{e:sing-perturb}, a first answer to problems (1)\&(2)   was
 provided,  still  in finite dimension, 
  in
 \cite{Zanini},
 whose results were later extended to second order systems in \cite{agostiniani2012}. 
 The  key assumptions are  that the energy $\mathcal{E} \in \mathrm{C}^3 ([0,T]\times \xfin)$
 \begin{itemize}
 \item[\emph{(i)}]
 has a \emph{finite} number of degenerate critical points,
  \item[\emph{(ii)}]
  the vector field $\mathrm{F} := \mathrm{D} \mathcal {E}$ complies with the so-called
 \emph{transversality conditions} at every degenerate critical point,
 \end{itemize}
 and 
  a further technical condition.
  While postponing to Section \ref{s:5}  
  a discussion on the 
 \emph{transversality conditions}, 
well-known in the realm of bifurcation theory
(see, e.g., \cite{guck-holmes, haragus,vanderbauwhede})
we may mention here that, essentially, they 
prevent degenerate  critical points from being ``too singular".
  \par
  Then, in \cite[Thm.\ 3.7]{Zanini} it was shown that, starting from a ``well-prepared'' datum $u_0$,
  there exists a
 unique piecewise $\mathrm{C}^2$-curve
$u:[0,T] \to \xfin$ with a finite jump set $J= \{t_1,\, \ldots, \,
t_k\}$, such that: 
\begin{enumerate}
\item
 $\mathrm{D} \ene t{u(t)} = 0$ with
$\mathrm{D}^2 \ene t{u(t)}  $ positive definite
 for all
$t\in [t_i, t_{i-1})$ and $i =1,\ldots,k-1$;  
\item
 at every jump point $t_i\in J$, the left limit 
$u_-(t_i)$ is a  degenerate critical point for
$\mathcal{E}_{t_i}(\cdot)$ and there exists a  unique curve $v\in
\mathrm{C}^2(\R;\xfin)$  connecting $u_-(t_i)$ to the right limit $u_+(t_i)$, in the sense that 
$  \lim_{s \to -\infty} v(s) = u_-(t_i),$ $  \lim_{s \to +\infty} v(s) =
   u_+(t_i)$,  and fulfilling  
    \begin{equation}
   \label{heterocline}
   v'(s) + \mathrm{D} \mathcal{E}_{t_i} (v(s)) =0 \quad \text{for all } s\in \R;
   \end{equation}
\item
 the \emph{whole} sequence $(u_\eps)_\eps$ converge to $u$ uniformly
 on the  compact sets of $[0,T] {\setminus} J$, and suitable
 rescalings of $u_\eps$ converge to $v$.
 \end{enumerate}
Let us stress that 
the fact that at each jump point $t_i$ the unique 
heterocline $v $  connecting  the left and the right limits   
$u_-(t_i)$ and $ u_+(t_i)$,
which is a gradient flow of the energy
$\mathcal{E}_{t_i}(\cdot)$,
does bear a mechanical  interpretation, akin to the one for solutions to rate-independent processes obtained in the vanishing-viscosity limit of
\emph{viscous} gradient systems, cf.\  \cite{EfMi06, MiRoSa09, MiRoSa, MRS13_prep}. 
   Namely, one observes that   the
internal scale of the system,
 neglected in the singular limit $\eps \down 0$,
  ``takes over'' and governs the dynamics in the jump regime,
   which can be in fact viewed as a  fast transition between two metastable
   states.
\par
The structure of the statement in \cite{Zanini} reflects the line of its proof.
 First, the unique limit curve is \emph{a priori}
constructed via the Implicit Function Theorem, also resorting to the
\emph{transversality} conditions. Secondly,  the convergence of
$(u_\eps)_\eps$ is proved.
  \paragraph{\bf Our results.} 
In this paper,  we  aim to extend the result from \cite{Zanini} 
to a wider class of energy functionals, still smooth
in the sense of \eqref{Ezero} but not necessarily of class $\rmC^3$, 
and not necessarily complying with the transversality conditions. 
To this end,
we will address the singular perturbation problem from a different perspective. 
Combining ideas from the \emph{variational approach} to gradient flows, possibly driven by  nonsmooth and nonconvex  energies, cf.\ 
\cite{AGS08, MRS13, RossiSavare06}, 
with the  techniques for the vanishing-viscosity approximation of  
rate-independent systems  from   \cite{MiRoSa, MRS13_prep},
we will  prove the existence of a limit curve by refined \emph{compactness}. 
\emph{Variational} arguments 
will  lead to   a  suitable energetic 
characterization of  its fast
dynamics at jumps. 
Indeed, the flexibility of this approach will allow us to extend the results obtained in this paper, to the infinite-dimensional setting, 
and to nonsmooth energies, in the forthcoming \cite{lavorone}. 

\par
The starting point for our analysis is the key observation that,
using equation \eqref{e:sing-perturb} to the rewrite the 
contribution $\int_s^t\ep\|u_\eps' (r)\|^2  \dd r $ of the dissipated energy,
the energy identity \eqref{enid-intro} 
can be reformulated as 
\begin{equation}
\label{enid-intro-better}
\int_s^t \left( \frac\eps 2 \|u_\eps' (r)\|^2 {+}\frac1{2\eps}
\|\mathrm{D} \mathcal{E}_r(u_\eps (r))\|^2\right) \dd  r +
\mathcal{E}_t (u_\eps(t)) = \mathcal{E}_s(u_\eps(s))+\int_s^t
\partial_t
\mathcal{E}_r(u_\eps(r))   \dd  r
\end{equation}
for all $0\leq s \leq t \leq T$. In addition to 
estimates \eqref{prelim-estimates-intro}, 
from \eqref{enid-intro-better}  it is possible to deduce that
\begin{equation}
\label{en-diss-bound} 
\int_0^T 
\|\mathrm{D}\mathcal{E}_r(u_\eps (r))\| 
\|u_\eps'(r)\|
\dd r \leq C.
\end{equation}
 Thus, while no (uniform w.r.t.\ $\eps>0$) bounds are available on
 $\|u_\eps'\|$,   estimate
 \eqref{en-diss-bound} suggests that: 
 \begin{itemize}
 \item[\emph{(i)}]
  The limit
 of the \emph{energy-dissipation integral} $\int_s^t
\|\mathrm{D} \mathcal{E}_r(u_\eps (r))\|
\|u_\eps' (r)\|\dd r $
 will describe the dissipation of energy (at jumps) in the  limit $\eps \down
 0$;
 \item[\emph{(ii)}]
  To extract compactness information from the integral \eqref{en-diss-bound},
with  the \emph{degenerating} weight
  $\|\mathrm{D}
\mathcal{E}_r(u_\eps(r) )\|$, it is necessary to suppose that the
(degenerate) critical points of $\mathcal{E}$, in whose neighborhood 
this weight tends to zero,
 are somehow ``well
separated'' one from each other.
\end{itemize}

In fact,  in addition 
to the aforementioned coercivity  and power control conditions  on $\calE$, typical of  the variational approach to existence for  non-autonomous gradient systems, in order to prove our  results 
for the singular  limit \eqref{e:sing-perturb} we will 
 resort to
the condition that for every $t\in [0,T] $ the
critical set
\begin{equation}
\label{non-variat} 
\mathcal{C}(t) := \{ u \in \xfin\, :  \ 
\mathrm{D} \mathcal{E}_t(u) =0  \}  \text{ consists of
\emph{isolated points}}.
\end{equation}
This  will allow  us  to prove in  {\bf \underline{Theorem \ref{mainth:1}}},
that, up to a subsequence, the gradient flows $(u_\eps)_\eps$  pointwise converge to a 
solution $u $ of the limit problem \eqref{e:lim-eq}, defined at \emph{every} $t\in [0,T]$,  enjoying  the following properties:
\begin{enumerate}
\item
$u:[0,T]\to\xfin$  is \emph{regulated}, 
i.e.\ the left and right limits $u_-(t)$ and $u_+(t)$   exist at every $t\in (0,T)$, and so do the limits $u_+(0)$ and $u_-(T)$;
 \item $u$ fulfills the energy balance
\begin{subequations}
\label{diss-visc-intro}
\begin{equation}
 \label{lim-enid}
\mu([s,t])+ \mathcal{E}_t(u_+(t)) = \mathcal{E}_s(u_-(s))+\int_s^t
\partial_t \mathcal{E}_r(u(r))   \dd  r \quad \text{for all } 0 \leq s \leq t \leq T,
\end{equation}
with $\mu$ a positive Radon measure with an at most countable set 
$J$ of atoms;
\item 
$u$ is continuous on $[0,T]\setminus J$, and solves 
\[
\rmD \ene t{u(t)} =0 \quad \text{in } \xfin \text{ for all } t \in [0,T]\setminus J;
\]
\item
$J$ coincides with the 
jump set of $u$,  
and there hold the jump relations
\begin{align}
\label{jump-relation}   \mu(\{t\}) = \mathcal{E}_t(u_-(t)) -
\mathcal{E}_t(u_+(t)) = \mathsf{c}(t; u_-(t), u_+(t) ) \quad \text{for all }
t \in J\,.
\end{align}
\end{subequations}
\end{enumerate}
In \eqref{jump-relation}, the cost function $\mathsf{c} : [0,T] \times \xfin \times \xfin \to [0,+\infty)$ is defined by   minimizing the energy-dissipation integrals, namely 
\[
\mathsf{c}(t; u_-, u_+ ): = 
\inf\left\{ \int_0^1
\|\mathrm{D} \mathcal{E}_t (\teta(s))\|\|\teta' (s)\| \dd s\, : \ \teta \in    \admis{u_-}{u_+}{t} \right \} \qquad \text{ for } t \in [0,T], \  u_-, \, u_+ \in \xfin,
\]
over a suitable class  $ \admis{u_-}{u_+}{t} $ of \emph{admissible curves} connecting $u_-$ and $u_+$. 
These curves somehow capture the asymptotic behavior of the gradient flows $(u_\eps)_\eps$ on intervals shrinking to the jump point $t$. 
The jump relations 
provide  a description of the behavior of the limit curves at jumps: indeed, it is possible to
deduce from \eqref{jump-relation}  that
any   curve $\teta$  attaining the infimum in the definition of 
$\mathsf{c}(t; u_-(t), u_+(t)  )$, hereafter referred to as \emph{optimal jump transition}, 
can 
 be reparameterized to a  curve $\tilde \teta$  solving the analogue of \eqref{heterocline}, namely
\begin{equation}
\label{analogy}
\tilde{\teta}'(\sigma) +\mathrm{D} \mathcal{E}_t(\tilde{\teta}(\sigma)) = 0 \quad \text{in } \xfin.
\end{equation}
Thus, 
the notion of solution to \eqref{e:lim-eq} given by (1)--(4), 
hereafter referred to as  \emph{Dissipative Viscosity} solution,  bears
the same mechanical interpretation as the solution concept in \cite{Zanini}. 

\par
Using the results of \cite{genericity}, 
we also show that our condition \eqref{non-variat} 
on the critical points 
can be deduced from the \emph{transversality conditions}
assumed in \cite{Zanini}. In turn, as we will see, 
these conditions have a generic character.
 
\par
Our second main result, {\bf \underline{Theorem \ref{mainth:2}}},   
shows that if $\calE$ fulfills the following condition
 \begin{equation}
 \label{subdiff-intro}
 \limsup_{v\to u} \frac{\ene tv - \ene tu}{\| \rmD \ene tv \|} \geq 0 \quad \text{ at every } u \in \calC(t) \quad \text{ for all } t \in [0,T],
\end{equation}
then for every Dissipative Viscosity solution   the absolutely continuous and the Cantor part of the associated defect  measure $\mu$ are zero.
Hence, $u$ improves to a \emph{Balanced Viscosity}  
(in the sense of \cite{MiRoSa09, MRS13_prep}) solution of \eqref{e:lim-eq}.  
Observe (cf.\ Remark \ref{rmk:Loja}), that a sufficient condition  for \eqref{subdiff-intro} is that 
$\calE$ complies with the celebrated \L ojasiewicz inequality,
cf., e.g., \cite{BoDaLe,Loja1,Simon1983}, as well as 
the recent survey paper \cite{Col_Min}.


\paragraph{\bf Plan of the paper.} 
In Section \ref{s:2}
we enucleate our conditions    on the energy functional  $\calE$,  and then give the definition of \emph{admissible curve} connecting two points  and the induced notion 
of \emph{energy-dissipation cost} $\mathsf{c}$. 
  We then introduce the two notions of
\emph{Dissipative Viscosity} and \emph{Balanced Viscosity} solutions to \eqref{e:lim-eq} and finally state \textbf{\underline{Theorems \ref{mainth:1} \& \ref{mainth:2}}}. 
In Section \ref{s:ojt}  we gain further insight into the properties of optimal jump transitions. 
Section \ref{s:3} is devoted to the analysis of the asymptotic  behavior of   the energy-dissipation integrals in the vanishing-viscosity limit, and to the properties of  the cost $\mathsf{c}$. These results lie at the core of the proof of Theorem  \ref{mainth:1}, developed in Section \ref{s:4} together with the proof of Theorem \ref{mainth:2}.  
In  Section \ref{s:5} we present examples of energies complying with our set of assumptions. 
In particular, on the one hand we show that \eqref{non-variat} is guaranteed by the \emph{transversality conditions}, whose genericity is discussed. On the other hand, we introduce the class of subanalytic functions, which comply with the Lojasiewicz inequality, hence with \eqref{subdiff-intro}. 

\paragraph{\bf Acknowledgment.} 
We are extremely grateful to Giuseppe Savar\'e for suggesting this problem to us and for several enlightening discussions and suggestions. 


\section{Main results}
\label{s:2}

Preliminarily, let us fix some \textbf{general notation} that will be used throughout. 
As already mentioned in the introduction, $X$ is a finite-dimensional 
Hilbert space (although all of the results of this paper could be trivially
extended to the, still finite-dimensional,  Banach  framework), 
with inner product
$\pairing{}{}{\cdot}{\cdot}$. 
Given $x\in \xfin$ and $\rho>0$, we will denote by
$B(x,\rho)$ the open ball centered at $x$ with radius $\rho$. 
\par
We will denote by $\mathrm{B}([0,T];\xfin)$ the class of measurable, \underline{everywhere defined}, 
and bounded functions from $[0,T]$ to $\xfin$, 
whereas
$\mathrm{M}(0,T)$ stands for the set of Radon measures on   $[0,T]$.
\par
Finally,  the symbols $c,\,C, \, C',\hdots$ will be used
to denote a positive constant depending on given data, and possibly
varying from line to line.


\paragraph{\bf Basic conditions on $\calE$.}
In addition to \eqref{Ezero}, we will require 
\begin{description}
\item[Coercivity]  the map  
$ u\mapsto \cg u: = \sup_{t \in [0,T]} | \ene tu|$ fulfills
\begin{equation}
  \label{coercivita}
\tag{${\mathrm{E}_1}$} 
\forall\, \rho>0 \ \text{ the sublevel set }
 \subl{\rho} :=
\{ u \in \xfin\, : \ \cg{u}\leq\rho\}
\text{ is bounded.}
   \end{equation}


\item[Power control]
the partial time derivative $\partial_t \ene tu=:\pt tu$ fulfills 
\begin{equation}
    \label{P_t}
\tag{${\mathrm{E}_2}$}
\begin{aligned}
  & \exists \,C_1,\, C_2>0 \ \ \forall\, (t,u)\in
 [0,T]
    \, : \ \ |\pt tu|\leq C_1\ene tu + C_2.
    \end{aligned}
  \end{equation}
\end{description}
Observe that \eqref{P_t} in particular yields that $\calE$ is bounded from below.
In what follows, without loss of generality we will suppose that $\calE$ is nonnegative. 
A simple argument based on the Gronwall Lemma ensures that  
\begin{equation}
\label{gronw-conseq}
\cg u 
\,\leq\,
\exp(C_1 T) \left(  \inf_{t \in [0,T]}
\ene tu  + C_2\,T\right)\qquad \text{for all } u \in \xfin. 
\end{equation}

\par
Under these conditions, the existence of solutions to the gradient flow 
\eqref{e:sing-perturb} is classical.
Testing \eqref{e:sing-perturb}  by $u'$ and using the chain rule fulfilled by the (smooth) energy $\calE$ leads to the energy identity \eqref{eqn_lemma} below, which will be the starting point in the derivation of all our estimates for the singular perturbation  limit as $\eps\down 0$.

\begin{theorem}
\label{thm:exist-g-flow} 
Let $\cE: [0,T]\times \xfin \to[0,+\infty)$
comply with \eqref{Ezero}, \eqref{coercivita}, and
\eqref{P_t}. Then,  for
every $u_0 \in \xfin$ there exists $u\in
H^1(0,T;\Hilbert) $, with $u(0)=u_0$,  solving   \eqref{e:sing-perturb} and
fulfilling for every $0\leq s \leq t \leq T$ the energy identity
\begin{equation}
\label{eqn_lemma}
\int_s^t\left(\frac{\ep}2\|u'(r)\|^2
+\frac1{2\ep}\|\rmD \ene r{u(r)} \|^2 \right)\dd r
+\mathcal E_t(u(t))=
\mathcal E_s(u(s))
+\int_s^t\pt r{u(r)}\dd r.
\end{equation}
\end{theorem}
\paragraph{\bf A condition on the critical points of $\calE$.}
In what follows,  we will denote the set of the critical points of 
$\ene t{\cdot}$, for fixed $t\in [0,T]$, by
\[
\calC(t): = \big\{ u \in \xfin\, : \ \rmD \ene tu=0\big\},
\]
and assume that
 \begin{equation}
  \label{strong-critical}
  \tag{$\mathrm{E}_3$}
  \begin{aligned}
\text{for every } t \in [0,T]  \quad \text{the set } \calC(t) \text{ consists of isolated points}.
    \end{aligned}
  \end{equation}
We postpone to Section \ref{s:5} a discussion on sufficient conditions  
for \eqref{strong-critical}, as well as on its \emph{generic} character. 


\paragraph{\bf Solution concepts.}
\label{ss:2.2-added}
We now illustrate the 
two notions of evolution of curves of critical points that we will obtain in the limit passage as $\eps\downarrow 0$. 
Preliminarily, we need to give 
the definitions of \emph{admissible curve} and of \emph{energy-dissipation cost}, obtained by minimizing the energy-dissipation integrals
along admissible curves. 
The latter notion somehow encodes the asymptotic properties of 
(the energy-dissipation integrals along) sequences of absolutely 
continuous curves (in fact, the solutions 
of our gradient flow equation), 
considered on intervals shrinking to a point $t\in [0,T]$, cf.\ Proposition \ref{super-lemmone} ahead. Basically, admissible curves are piecewise locally Lipschitz curves joining critical points. Note however that we do not impose that their end-points be critical. That is why, we choose to confine our definition to the case the end-points are different: otherwise, we should have to allow for curves degenerating to a single, possibly non-critical, point, which would not be consistent with \eqref{def-admissible-class-simpler} below. 

\begin{definition}
\label{main-def-1}
Let $t \in [0,T]$ and $u_1,\, u_2 \in \xfin$ be fixed.
\begin{enumerate}
\item
In the case  $u_1\neq u_2$, we call a curve 
$\teta\in\mathrm{C}
([0,1];\Hilbert)$
with $\teta(0) = u_1$ and $\teta(1) =
u_2$ \emph{admissible}  if there exists a partition 
$0=\mathsf{t}_0 < \mathsf{t}_1< \ldots <
\mathsf{t}_j=1$ such that 
\begin{equation}
\label{def-admissible-class-simpler}
\begin{aligned}
  &
       \teta|_{(\mathsf{t}_i,\mathsf{t}_{i+1})} \in
\mathrm{C}_{\mathrm{loc}}^{\mathrm{lip}}
((\mathsf{t}_i,\mathsf{t}_{i+1});\Hilbert) \text{ for all }
i=0,\ldots, j-1,
\\
& 
 \teta(\mathsf{t}_i) \in \mathcal{C}(t) \text{ for all } i\in\{1,\ldots, j-1\},
 \quad
\teta(r) \notin \mathcal{C}(t) \
 \forall\, r \in (\mathsf{t}_i,\mathsf{t}_{i+1}) \text{ for all } i=0,\ldots, j-1.
\end{aligned}
\end{equation}
We will denote by $ \admis{u_1}{u_2}{t} $ the class of admissible curves connecting $u_1$ and $u_2$ at time $t$. 
Furthermore, for a given $\rho>0$ we will use the notation
\[
 \admis{u_1}{u_2}{t,\rho}: = 
 \Big\{ \teta \in \admis{u_1}{u_2}{t} \, : \ \teta(s) \in \subl \rho \text{ for all } s \in [0,1]\Big\}.
\]
\item We define the \emph{energy-dissipation cost}
\begin{equation}
\label{costo-simpler}
\cost t{u_1}{u_2}:=
\left\{
\begin{array}{ll}
\inf\left\{\int_0^1 \minpartial \calE t {\teta(s)} \| \teta'(s)\|\dd
s \,:\,\teta\in
\admis{u_1}{u_2}{t}\right\} & \quad\mbox{if }\quad u_1\neq u_2,\
\\
0 & \quad\mbox{if }\quad u_1= u_2. 
\end{array}
\right.
\end{equation}
\end{enumerate}
\end{definition}
We call 
$\int_0^1 \minpartial \calE t {\teta(s)} \| \teta'(s)\|\dd s$,
for some 
$\teta\in\admis{u_1}{u_2}{t}$,
 \emph{energy-dissipation integral}.
Observe that, up to a reparameterization, every absolutely continuous curve
$\teta\in\mathrm{AC}([a,b];\Hilbert)$ such that $
\exists\,\,\,a=\mathsf{t}_0 < \mathsf{t}_1< \ldots < \mathsf{t}_j=b
$ with
\[
 \teta(\mathsf{t}_i) \in \mathcal{C}(t), \qquad   \teta(\mathsf{t}_i) \neq  \teta(\mathsf{t}_j), \qquad 
\teta(r) \notin \mathcal{C}(t) \
 \forall\, r \in (\mathsf{t}_i,\mathsf{t}_{i+1}) 
 \]
 for all $i,\, k \in \{1, \ldots, j-1\}$, $i\neq k$, 
is an admissible curve.
Note that the chain-rule holds along
\emph{admissible} curves $\teta$
with finite energy-dissipation integral at time $t$. 
This is the content of the following lemma,
which can be easily proved.

\begin{lemma}
\label{chain_facile}
Let
$t\in[0,T]$ be fixed, and
$\teta\in \admis{u_1}{u_2}{t}$ 
be an admissible curve connecting $u_1$ and $u_2$ such that
\begin{equation*}
\int_0^1 \minpartial \calE t {\teta(s)} \| \teta'(s)\|\dd s<\infty.
\end{equation*}
Then,
the map $s\mapsto \ene t{\teta(s)}$ belongs to
$\AC([0,1])$ and there holds
the chain rule
\begin{equation}
\label{ch-rule-admiss}
\begin{gathered}
\frac{\mathrm{d}}{\mathrm{d}s}\ene t{\teta(s)}
=
\langle\rmD \ene t{\teta(s)}, \teta'(s) \rangle
 \quad \text{ for a.a.} \, s\in  (0,1)\,.
\end{gathered}
\end{equation}
\end{lemma}

\par
The following result, whose proof is postponed to Section \ref{s:3}, 
collects the properties of the cost
$\costname{}{}$.

\begin{theorem}
\label{prop:cost}
Assume \eqref{Ezero}--\eqref{strong-critical}. 
Then, for every $t \in [0,T]$
and $u_1,\, u_2\in\xfin$ we have:
\begin{compactenum}
\item 
$\cost{t}{u_1}{u_2} =0$ if and only if 
$u_1 = u_2 $;
\item $\costname{t}$ is symmetric;
\item if  $\cost{t}{u_1}{u_2} >0$,    
there exists an optimal curve $\teta\in\admis{u_1}{u_2}{t}$ 
attaining the  $\inf$ in \eqref{costo-simpler};
\item for every $u_3\in\mathcal C(t)$,
the triangle inequality holds
\begin{equation}
\label{prop:triangle}
\cost{t}{u_1}{u_2}
\leq
\cost{t}{u_1}{u_3} + \cost{t}{u_3}{u_2};
\end{equation}
\item there holds
\begin{equation}\label{charact-cost_AC}
\begin{aligned}
   \cost{t}{u_1}{u_2}\leq\inf\Big\{ & \liminf_{n \to
\infty}\int_{t_1^n}^{t_2^n}\|\rmD\mathcal
E_s(\teta_n(s))\| \|\teta_n'(s)\| \dd s\,:\\ & \teta_n\in
\AC([t_1^n,t_2^n];X),\ t_i^n\to t, \ \teta_n(t_i^n)\to u_i \text{ for }i=1,2
\Big\}\,;
\end{aligned}
\end{equation} 
\item the following lower semicontinuity property holds
\begin{equation*}
  (u_1^k, u_2^k)\to (u_1,u_2)
\text{ as } k \to \infty
\quad\Rightarrow\quad
\liminf_{k \to \infty}
 \cost{t}{u_1^k}{u_2^k} 
\,\geq\,
\cost{t}{u_1}{u_2}.
\end{equation*}
\end{compactenum}
\end{theorem}
\par
We are now in the position to give the definition of \emph{Dissipative Viscosity} solution to equation \eqref{e:lim-eq}. 

\begin{maindefinition}[Dissipative Viscosity solution]
\label{def:sols-notions-1}
We call \emph{Dissipative Viscosity solution} to \eqref{e:lim-eq}  a
curve $u\in \mathrm{B}( [0,T] ; \Hilbert)$ such that
\begin{enumerate}
\item 
for every $0\leq t<T$ and every $0<s\leq T$,
the left and right limits
$ u_-(s) : = \lim_{\tau\uparrow s} u(\tau)$
and 
$u_+(t) : = \lim_{\tau\downarrow t} u(\tau)$
exist,
there exists a positive Radon measure $\mu \in \mathrm{M}(0,T)$ 
such that the set  $J$  of its atoms is countable, 
and $(u,\mu)$ fulfill the energy identity
\begin{align}
\label{to-save-eneq}
\mu([s,t])  +\ene t{u_+(t)}  &=\ene s{u_-(s)} + \int_s^t \pt r{u(r)} \dd r    \quad \text{for all }  0\leq s \leq t \leq T,
\end{align}
where we understand $u_-(0): = u(0)$ and $u_+(T): = u(T)$; 
\item
$u$ is  continuous on the set $[0,T]\setminus J$, and
\label{to-save-def}
\begin{align}
&
\label{to-save-1}
u(t) \in \calC(t)  \qquad \text{for all }  t\in (0,T]\setminus J;
\end{align}
\item  The left and right limits fulfill
\begin{subequations}
\label{to-save-complex}
\begin{align}
&
\label{to-save-2}
\begin{aligned}
&
 u_-(s)\in\calC(s)\text{ and }u_+(t) \in \calC(t)
 \ \text{ at every } 0<s\leq T \text{ and } 
 0\leq t< T,
\end{aligned}
\\
&
\label{to-save-3}
J  = \{ t \in [0,T] \, : \  u_-(t)  \neq  u_+(t)
  \},
\\
&
\label{to-save-4}
0 < \cost t {u_-(t)}{u_+(t)}   = \mu (\{ t\})  = \ene t{u_-(t)} -  \ene t{u_+(t)},  \quad \text{ for every } t \in J.
\end{align}
\end{subequations}
\end{enumerate}
\end{maindefinition}

A comparison between the energy balances 
\eqref{eqn_lemma} and
\eqref{to-save-eneq}
highlights  the fact that 
the contribution to \eqref{to-save-eneq}
given by   the  measure $\mu([s,t])$   surrogates the role of the energy-dissipation integral 
$
\int_s^t
\|\rmD \ene r{u(r)}\|
\|u'(r)\|
\dd r\,.
$
That is why, 
 in what follows we will refer to $\mu$ as the \emph{defect energy-dissipation measure} (for short, \emph{defect measure}), associated with $u$. Let us highlight that, by
\eqref{to-save-3}, $u$ jumps  at the atoms of $\mu$, 
and that the jump conditions  \eqref{to-save-4}  provide  a description of its energetic behavior in the jump regime
(see Propositions \ref{prop:insight} and \ref{prop:insight2}).

\par
The notion of \emph{Balanced Viscosity} solution below
brings the additional information that the measure $\mu$  is purely atomic.
Then, taking into account 
conditions \eqref{to-save-4}, we obtain
\[
\mu([s,t]) = \sum_{r\in J \cap [s,t]} \mu(\{r\}) =  \sum_{r\in J \cap [s,t]} \cost r{u_-(r)}{u_+(r)}\,.
\]
This results in a more transparent form of the energy  
balance \eqref{to-save-eneq},
cf.\ \eqref{to-save-enineq-BV} below,
akin to the one featuring  in the notion of Balanced Viscosity solution 
to a rate-independent system, cf.\ \cite{MRS13_prep}.
\begin{maindefinition}[Balanced Viscosity solution]
\label{def:bv-intro}
We call a Dissipative Viscosity solution
$u$ to \eqref{e:lim-eq}
 \emph{Balanced Viscosity} solution if the absolutely continuous and the Cantor part of the defect measure $\mu$ are zero.  Therefore,
\eqref{to-save-eneq}
 reduces to
\begin{equation}
\label{to-save-enineq-BV}
\begin{aligned}
\sum_{r\in J \cap [s,t]} \cost r{u_-(r)}{u_+(r)}  +  
\ene t{u_+(t)} = 
\ene s {u_-(s)}+ \int_s^t \pt r{u(r)} \dd r 
\text{ for every } 
0 \leq s \leq t \leq T. 
\end{aligned}
\end{equation}
\end{maindefinition}

\paragraph{\bf Convergence to Dissipative Viscosity solutions.}
Our \textbf{first main result}, 
whose proof will be given throughout Sections \ref{s:3} \& \ref{s:4}, 
ensures the convergence, up to a subsequence, of any family of solutions to (the Cauchy problem for) \eqref{e:sing-perturb}, to 
a Dissipative Viscosity solution. 
\begin{maintheorem}
\label{mainth:1}
Assume \eqref{Ezero}--\eqref{strong-critical}. Let $(\eps_n)_n$ be a null sequence, and consider a sequence 
$(u_{\eps_n}^0)_n$ of initial data for  \eqref{e:sing-perturb} such that 
\begin{equation}
\label{cvg-initial-data}
u_{\eps_n}^0 \to u_0 \quad \text{as } n\to\infty.
\end{equation}
Then there exist a (not relabeled) subsequence and a curve $u \in \mathrm{B}([0,T];\xfin)$ such that 
\begin{enumerate}
\item the following convergences hold
\begin{align}
&
\label{3converg_var}
u_{\ep_n}(t)\to u(t) \quad \text{for all } t \in [0,T] 
\\
&
\label{4converg_var}
u_{\ep_n}\weaksto u \quad \text{in $L^\infty(0,T;X)$,}\qquad u_{\ep_n}\to u \quad \text{in $L^p(0,T;X)$
\ for all\ $1\leq p<\infty$;}
\end{align}
\item
$u(0)=u_0$ and $u$ is  Dissipative Viscosity solution to \eqref{e:lim-eq}.
\end{enumerate}
\end{maintheorem}

We now address the improvement of
Dissipative Viscosity to Balanced Viscosity solutions, 
under the condition that 
at every $t\in [0,T]$ there holds
\begin{equation}
\label{Loja}
\tag{$\mathrm{E}_4$}
 \limsup_{v\to u} \frac{\ene tv -\ene tu}{\| \rmD \ene tv\|}\geq 0 \quad \text{for all } u \in \calC(t).
\end{equation} 

\par
The proof of our \textbf{second  main result} is also postponed to Section \ref{s:4}. 

\begin{maintheorem}
\label{mainth:2}
In the setting of \eqref{Ezero}--\eqref{P_t}, assume in addition 
\eqref{Loja}. 
Let $u$ be a \emph{Dissipative Viscosity} solution to  \eqref{e:lim-eq} and let $\mu$ be its associated defect measure.
Then, the absolutely continuous part
$ \mu_{\mathrm{AC}}$
and the Cantor part
$\mu_{\mathrm{Ca}}$
of the measure $\mu
$
are zero, i.e.\ $u$ is a   \emph{Balanced Viscosity}
 solution to  \eqref{e:lim-eq}.
 \end{maintheorem} 

\begin{remark}[A discussion of  \eqref{Loja}]
\label{rmk:Loja}
\upshape
Observe
that \eqref{Loja} is trivially satisfied  in the case the functional $u\mapsto \ene tu $ is \emph{convex}. Indeed, if $\rmD \ene tu=0$, 
then $\ene tv \geq \ene tu $ for all $v\in \xfin$. 
\par
Another sufficient condition for  \eqref{Loja} is that 
$\calE$ complies with the celebrated {\L}ojasiewicz inequality, namely
\begin{equation}
\label{true-Loja}
\! \! \!  \forall\, t \in [0,T]\  
\forall\, u \in \calC(t)
\ \exists\, \theta \in (0,1) \  \exists\, C>0 \  \exists\, R>0 \  \forall\, v \in B_R( u)\, :
\ |\ene tv - \ene t{ u}|^\theta \leq C \|\rmD \ene tv\|.
\end{equation}
In this case, we even have
\[
\lim_{v\to u}\frac{\left|\ene tv-\ene tu\right|}{\| \rmD \ene tv\|}\leq C \lim_{v\to u}  |\ene tv - \ene t{ u}|^{1-\theta}=0
\]
by continuity of $u\mapsto \ene tu$.
\end{remark} 

\section{Optimal jump transitions}
\label{s:ojt}
In this section, we get 
further insight into the jump conditions \eqref{to-save-4}.
Due to Theorem \ref{prop:cost} (3), 
for every $t \in J$ the left and right limits 
$u_-(t)$ and $u_+(t)$ are connected by a curve
 $\teta \in \admis{u_-(t)}{u_+(t)}{t}$    minimizing
the cost $\cost t{u_-(t)}{u_+(t)}$,
which will be hereafter referred to as an \emph{optimal jump transition} between $u_-(t)$ and $u_+(t)$. 
The following result states that every 
$\mathrm{C}_{\mathrm{loc}}^{\mathrm{lip}}$-piece of an optimal jump transition can be reparameterized to a curve 
solving the gradient flow equation 
\eqref{rescaled-gradient-flow} below.
 
\begin{prop}
\label{prop:insight}
Let $u\in \mathrm{B}( [0,T] ; \Hilbert)$ 
be a \emph{Dissipative Viscosity solution} to \eqref{e:lim-eq}.
Also, let  $t \in J$  be fixed,  and let $\teta \in \admis{u_-(t)}{u_+(t)}{t}$  
be an optimal jump transition  between $u_-(t)$ and $u_+(t)$; 
let $(\mathsf{a},\mathsf{b}) \subset [0,1] $ be such that  $\teta|_{(\mathsf{a},\mathsf{b})} \in
\mathrm{C}_{\mathrm{loc}}^{\mathrm{lip}}
((\mathsf{a},\mathsf{b});X)$ and $\teta(s) \notin\mathcal{C}(t)$ for all $s \in (\mathsf{a},\mathsf{b}) $.
Then, there exists a reparameterization $\sigma \mapsto \mathfrak{s}(\sigma)$ mapping
a (possibly unbounded) interval $(\tilde{\mathsf{a}},\tilde{\mathsf{b}})$ into $(\mathsf{a},\mathsf{b})$, such that 
the curve $ \tilde{\teta}(\sigma):= \teta (\mathfrak{s}(\sigma))$ 
is locally absolutely continuous and 
fulfills the gradient flow equation 
\begin{equation}
\label{rescaled-gradient-flow}
{\tilde\teta}'(\sigma) +  \rmD\mathcal
E_t(\tilde\teta(\sigma)) =0 
\quad \foraa \sigma \in (\tilde{\mathsf{a}},\tilde{\mathsf{b}})\,.
\end{equation}
\end{prop}
\begin{proof} 
Any optimal jump transition $\teta \in \admis{u_-(t)}{u_+(t)}{t}$ 
fulfills the jump condition \eqref{to-save-4}
with $u_-(t)=\teta(0)$ and $u_+(t)=\teta(1)$. 
Combining this with the chain rule
(see Lemma \ref{chain_facile}), 
we conclude that
\begin{equation}
\label{punto-di-partenza}
\begin{aligned}
\frac{\mathrm{d}}{\mathrm{d}s}\mathcal E_t (\teta(s))=
\langle\rmD\mathcal
E_t(\teta(s)),\teta'(s) \rangle= -
  \|\rmD\mathcal
E_t(\teta(s))\| \|\teta'(s)\|
\quad \foraa\, s \in (\mathsf{a},\mathsf{b}),
\end{aligned}
\end{equation}
hence 
\begin{equation}
\label{punto-di-partenza-2}
\foraa\, s \in (\mathsf{a},\mathsf{b}) \quad \exists\, \lambda(s) >0 \, : \quad 
\lambda(s) \teta'(s) +
\rmD\mathcal
E_t(\teta(s))
=0. 
\end{equation}
In order to find $\mathfrak{s}=\mathfrak{s}(\sigma)$, 
we fix  $\bar s \in (\mathsf{a},\mathsf{b})$ and set 
\begin{equation}
\label{def-sigma}
\sigma(s):= \int_{\bar s}^s \frac{1}{\lambda(r)} \, \mathrm{d}r\,.
\end{equation}
Indeed, it follows from \eqref{punto-di-partenza-2} that $\lambda(s) =\frac{ \| \rmD\mathcal
E_t(\teta(s)) \|}{\|\teta'(s)\| } $ for almost all $s \in (\mathsf{a},\mathsf{b}) $. Since $s \in (\mathsf{a},\mathsf{b}) \mapsto \|\rmD\mathcal
E_t(\teta(s)) \|$ is strictly positive and continuous,
and since $\teta$ is locally Lipschitz on $(\mathsf{a},\mathsf{b}) $,
it is immediate to deduce that 
for every closed interval 
$[{\mathsf{a}}+\rho,\mathsf{b}-\rho]\subset ({\mathsf{a}},\mathsf{b})$ there exists
$\lambda_\rho>0$ such that $\lambda(s) \geq \lambda_\rho$ for all 
$s \in [{\mathsf{a}}+\rho,\mathsf{b}-\rho]$.  Therefore, $\sigma$
is a well-defined, locally Lipschitz continuous map 
with $\sigma'(s)>0$ 
for almost all $s \in (\mathsf{a},\mathsf{b})$. 
We let $\tilde{\mathsf{a}}:= \sigma(\mathsf{a})$, $\tilde{\mathsf{b}}:= \sigma(\mathsf{b})$, and  set
$\mathfrak{s}: (\tilde{\mathsf{a}},\tilde{\mathsf{b}})\to  (\mathsf{a},\mathsf{b}) $ to 
be the inverse map of $\sigma$: it satisfies 
\begin{equation}\label{form-deriv}
\mathfrak{s}'(\sigma)= \lambda (\mathfrak{s}(\sigma)) \quad \foraa\, \sigma \in (\tilde{\mathsf{a}},\tilde{\mathsf{b}}),
\end{equation}
and it is an absolutely continuous map, being 
$  \int_{\tilde a}^{\tilde b}\mathfrak s'(\sigma)d\sigma  =\int_{\tilde a}^{\tilde b}\lambda(\mathfrak s(\sigma))d\sigma=b-a$. 
Using the definition of $\tilde\teta$, \eqref{form-deriv}, and  \eqref{punto-di-partenza-2}, we conclude that $\tilde \teta$ fulfills 
\eqref{rescaled-gradient-flow}.
Since $\mathfrak s$ is absolutely continuous and $\teta$ locally Lipschitz continuous, the
curve $\tilde\teta$ turns out to be locally absolutely continuous. 
\end{proof}

The symmetry property of the cost
proved in Theorem \ref{prop:cost} (2)
gives some information about the number of the optimal jump transitions.
This is the content of the following proposition.

\begin{prop}
\label{prop:insight2}
Let $u\in \mathrm{B}( [0,T] ; \Hilbert)$ 
be a \emph{Dissipative Viscosity solution} to \eqref{e:lim-eq},
and let $t \in J$.
There exists a finite number of optimal jump transitions between 
$u_-(t)$ and $u_+(t)$.
\end{prop}

\begin{proof}
Suppose by contradiction that there exists an infinite number of optimal jump
transitions connecting $x:=u_-(t)$ and $y:=u_+(t)$ and, 
for an arbitrary natural number $N$, choose 
$2N+1$ of them: $\teta_1,...,\teta_{2N+1}$.  
Let us fix an arbitrary partition $0=t_0 < \ldots < t_{2N+1}=1$. 
We can suppose that,
up to reparametrizations, 
$\teta_{2i+1}:[t_{2i},t_{2i+1}]\to X$ is such that $\teta_{2i+1}(t_{2i})=x,\,\teta_{2i+1}(t_{2i+1})=y$
for $i=0,...,N$,
and $\teta_{2i}:[t_{2i-1},t_{2i}]\to X$ is such that $\teta_{2i}(t_{2i-1})=y,\,\teta_{2i}(t_{2i})=x$ 
for $i=1,...,N$.
Consider the function $\teta:[0,1]\to X$ defined as
\begin{equation*}
\teta:=
\begin{cases}
\teta_{2i+1} & \text{on }[t_{2i},t_{2i+1}]\text{ for }i=0,...,N, \\
\teta_{2i} & \text{on }[t_{2i-1},t_{2i}]\text{ for }i=1,...,N,
\end{cases}
\end{equation*}
and note that $\teta(0)=x$ and $\teta(1)=y$.
Therefore, by the chain rule we have
\begin{eqnarray*}
+\infty>\ene tx-\ene ty & = & -\sum_{i=0}^{2N}\int_{t_i}^{t_{i+1}}
\la\rmD\ene t{\teta_{i+1}(s)},\teta'_{i+1}(s)\ra \dd s\\
 & = & \sum_{i=0}^{2N}\int_{t_i}^{t_{i+1}}
\|\rmD\ene t{\teta_{i+1}(s)}\|
\|\teta_{i+1}'(s)\|\dd s=(2N+1)\cost txy,
\end{eqnarray*}
where the second equality is due to the fact that $\teta_{i+1}$ 
is an optimal jump transition on $[t_i,t_{i+1}]$
for $i=0,...,2N$ (cf. \eqref{punto-di-partenza}), 
and the third descends from the symmetry of the cost.
Since $N$ can be chosen arbitrarily large, the above equalities give a contradiction.
\end{proof}

In what follows 
we show that, if the energy $\cE$ complies with 
the {\L}ojasiewicz inequality \eqref{true-Loja}
(which  implies  \eqref{Loja},
as observed in Remark \ref{rmk:Loja}), 
the optimal jump transitions connecting jump points of
Balanced Viscosity solutions 
(see Definition \ref{def:bv-intro} and Theorem \ref{mainth:2})
have a further property. 
In fact, they have finite length.

\begin{theorem}
\label{th:6.1}
In the setting of \eqref{Ezero}--\eqref{P_t}, 
assume in addition \eqref{true-Loja},
and let $u$ be a \emph{Balanced Viscosity} solution to  
\eqref{e:lim-eq}.
For $t\in [0,T]$ fixed, let   $\teta \in \admis{u_-(t)}{u_+(t)}{t}$  be an
optimal jump transition  between $u_-(t)$ and $u_+(t)$, and let
$(\mathsf{a},\mathsf{b}) \subset [0,1] $ be such that
$\teta|_{(\mathsf{a},\mathsf{b})}\in
\mathrm{C}_{\mathrm{loc}}^{\mathrm{lip}}
((\mathsf{a},\mathsf{b});\Hilbert)$,
$\teta(s)\notin\mathcal{C}(t)$ for all
$s \in (\mathsf a,\mathsf b)$,
and $\teta(a)$, $\teta(b)\in\calC(t)$.
Then the curve $\teta|_{(\mathsf{a},\mathsf{b})}$ has finite length.
\end{theorem}

For the proof of Theorem \ref{th:6.1} we shall exploit   the crucial
fact that, since $\teta $ is an \emph{optimal jump transition}, (a
reparameterization of) $\teta|_{(\mathsf{a},\mathsf{b})}$ is a
gradient flow of the energy $\cE_t$, cf.\ Proposition \ref{prop:insight}.  
This will allow us to develop
arguments for gradient systems driven by energies satisfying the
{\L}ojasiewicz inequality,
showing that the related trajectories 
have finite length.

\begin{proof}
Recall from Proposition \ref{prop:insight}
that $\teta|_{(\mathsf{a},\mathsf{b})}$ 
can be reparameterized to a curve
$\tilde\teta$ on a (possibly unbounded)
interval $(\tilde{\mathsf a},\tilde{\mathsf b} )$
such that $\tilde \teta$ fulfills the gradient flow equation
\begin{equation}
\label{more-precise-gflow_smooth}
\tilde{\teta}'(\sigma) + \argminpartial {\cE}
 t{\tilde\teta(\sigma)} =0 
\qquad \foraa\, \sigma \in (\tilde{\mathsf a},\tilde{\mathsf b} ).
\end{equation}
Observe that for every $R>0$ there exists 
$\sigma_R>\tilde{\mathsf a}$ such that
\begin{equation}
\label{eq:per_lungh_Loja_1}
\tilde\teta(\sigma)\in B_R(\teta(b))
\qquad\mbox{for every}\quad\sigma>\sigma_R.
\end{equation}
Supposing for simplicity that $\cE_t(\teta(b))=0$,
observe that \eqref{true-Loja} reads
\begin{equation}
\label{eq:per_lungh_Loja_2}
\big(\ene tv \big)^\theta \leq C \minpartial {\cE}
 t{v}
\qquad\mbox{for all}\quad v\in B_R(\teta(b)).
\end{equation}
From \eqref{more-precise-gflow_smooth},
\eqref{eq:per_lungh_Loja_1} and \eqref{eq:per_lungh_Loja_2}
we deduce that for every $\tilde\sigma\in(\sigma_R,\tilde{\mathsf b})$
it holds
\begin{align*}
\int_{\sigma_R}^{\tilde\sigma}
\|\tilde\teta'(\sigma)\|\dd\sigma
&=
\int_{\sigma_R}^{\tilde\sigma}
\minpartial {\cE} t{\tilde\teta(\sigma)} \dd\sigma
\leq
C\int_{\sigma_R}^{\tilde\sigma}
\frac{\minpartial {\cE} t{\tilde\teta(\sigma)}^2}
{\big(\ene t{\tilde\teta(\sigma)}\big)^{\theta}}\dd\sigma\\
&=
-\frac C{1-\theta}
\left[\big(\ene t{\tilde\teta(\tilde\sigma)}\big)^{1-\theta}
-
\big(\ene t{\tilde\teta(\sigma_R)}\big)^{1-\theta}
\right]\leq C'. 
\end{align*}
Note that in the second equality we have used the
fact that
\[
\frac{\dd}{\dd\sigma}\big(\ene t{\tilde\teta(\sigma)}\big)^{1-\theta}
=
-(1-\theta)\big(\ene t{\tilde\teta(\sigma)}\big)^{-\theta}
\minpartial{\cE} t{\tilde\teta(\sigma)}^2,
\]
cf.\ also \eqref{punto-di-partenza}.
In particular, we have obtained that
$\int_{s_{\mathsf b,R}}^{\mathsf b}
\|\teta'(s)\|\dd s<\infty$,
for some $s_{\mathsf b,R}\in(\mathsf a,\mathsf b)$.
Arguing in a similar way, one can obtain that
$\int_{\mathsf a}^{s_{\mathsf a,R}}
\|\teta'(s)\|\dd s<\infty$ as well,
for some $s_{\mathsf a,R}\in(\mathsf a,\mathsf b)$,
and this finishes the proof.
\end{proof}


\section{Properties of the energy-dissipation integrals and cost}
\label{s:3}
In order to prove  Thm.\ \ref{prop:cost}  on properties of the cost function $\mathsf{c}$, it is necessary to investigate the limit of the energy-dissipation integrals 
$\int_0^1 \minpartial \calE t {\teta(s)} \| \teta'(s)\|\dd
s$, which enter in the definition \eqref{costo-simpler} of $\costname t$, along sequences of (admissible) curves. 
This section collects all the technical results underlying the proof of convergence to Dissipative Viscosity solutions. 
\par
With our first result, Proposition \ref{super-lemmone},  
we gain insight into the asymptotic behavior of the 
energy-dissipation integrals 
$\int_{t_1^n}^{t_2^n}
\minpartial{\cE}{r}{\teta_n(r)}\,\mdt{\teta} n r\dd r$, where the curves $(\teta_n)_n$ are defined on intervals $[t_1^n,t_2^n]$ shrinking to a singleton $\{ t\}$, whereas in the integrand
$\minpartial{\cE}{\cdot}{\teta_n(\cdot)}\,\mdt{\teta} n {\cdot}$  
the time variable is not fixed. Both Proposition 
 \ref{super-lemmone} 
and a variant of it, 
Proposition \ref{lemma:extension-lemmone},  
will be at the core of the proof of Thm.\  \ref{prop:cost}. 
Their proof is based on a  reparameterization technique, 
combined with careful compactness arguments for the reparameterized
curves. 

\begin{proposition}
\label{super-lemmone}
Assume  \eqref{Ezero}--\eqref{strong-critical}.
Let $t\in[0,T]$, $\rho>0$, and $u_1$, $u_2\in\xfin$ be fixed and let
$(t_1^n)_n$, $(t_2^n)_n$, with $0\leq t_1^n \leq t_2^n \leq T$ for every $n\in \N$, and $(\teta_n)_n\subset\AC
([t_1^n,t_2^n];\Hilbert)$ 
fulfill 
\begin{equation}
\label{new-cost}
 t_1^n,\, t_2^n \to t,  \quad
\teta_n(t_1^n) \to u_1, \ \teta_n(t_2^n) \to u_2\,. 
 \end{equation}
Then, the following implications hold:
\begin{compactenum}
\item
If
\begin{equation}
\label{liminf-zero}
\liminf_{n \to \infty}\int_{t_1^n}^{t_2^n}
\minpartial{\cE}{r}{\teta_n(r)}\,\mdt{\teta} n r \dd r=0,
\end{equation}
then $u_1=u_2$;
\item
In the case $u_1 \neq u_2$, so that 
\begin{equation}
\label{lemmone:caseI}
\liminf_{n \to \infty}\int_{t_1^n}^{t_2^n}
\minpartial{\cE}{r}{\teta_n(r)}\,\mdt{\teta} n r \dd r>0,
\end{equation}
there exists $\teta \in\admis{u_1}{u_2}{t}$ 
such that
\begin{equation}
\label{liminf-lemmone}
\liminf_{n \to \infty}
\int_{t_1^n}^{t_2^n}\minpartial{\cE}{r}{\teta_n(r)}\,\mdt{\teta} n r  \dd r
\,\geq\,
\int_0^1 \minpartial {\cE}{t}{\teta(s)}\,
\|\teta' (s)\| \dd s.
\end{equation}
\end{compactenum}
\end{proposition}
Preliminarily, we need the following result.
\begin{lemma}
\label{lemma:auxi-v} 
Let $ K$ be a closed subset of
some sublevel set  $ \subl{\rho}$, with $\rho>0$, 
and suppose that for some $t \in (0,T)$
\begin{equation}\label{basso_unif}
\inf_{u\in K} \minpartial \calE tu>0.
\end{equation}
Then, the $\inf$ in \eqref{basso_unif} is attained, and there exists
$\alpha=\alpha(t)>0$ such that
\begin{equation}
\label{basso-unif-time} \min_{u\in
K,s\in[t-\alpha,t+\alpha]}\minpartial \calE su>0\,.
\end{equation}
\end{lemma}
\begin{proof}
It follows from \eqref{coercivita} that $K$ is compact, therefore
  $\inf_{u \in K}
\minpartial \calE su  $ is attained for every $s \in [0,T]$, and
 the function $s \mapsto \min_{u \in K} \minpartial \calE su$ is
continuous since $\calE \in \rmC^1 ([0,T]\times \xfin).$  Combining this fact with \eqref{basso_unif},
we conclude \eqref{basso-unif-time}.
\end{proof}
We are now in the position to develop the
\underline{\bf proof of Proposition \ref{super-lemmone}}:  
preliminarily, we observe that there exists $\rho>0$ such that the curves $(\teta_n)_n$ in \eqref{new-cost} fulfill
\begin{equation}
\label{in-a-sublevel}
\teta_n([t_1^n,t_2^n]) \subset \subl\rho \quad \text{for every } n \in \N\,.
\end{equation}
Indeed, 
since $\ene 0{\cdot}$ is continuous, from $ \teta_n(t_i^n) \to u_i$ for $i=1,2$ we deduce that
 $\sup_n  |\ene 0{\teta_n(t_1^n)}| +  |\ene{0}{\teta_n(t_2^n)}| \leq C$.  Hence, $\sup_n \calG(\teta_n(t_1^n)) +  \calG(\teta_n(t_2^n)) \leq C'$ in view of \eqref{gronw-conseq}. We now apply the chain rule along the curve $\teta_n$ to conclude that 
 \[
 \ene t{\teta_n(t)} \leq   \ene{t_1^n}{\teta_n(t_1^n)} + \int_{t_1^n}^t \partial_t \ene{r}{\teta_n(r)} \dd r +  \int_{t_1^n}^{t_2^n} 
 \minpartial \calE r{\teta_n(r)} \| \teta_n'(r)\| \dd r.
 \]
From the above estimate, we immediately conclude via the power estimate \eqref{P_t}, the Gronwall Lemma, and condition \eqref{liminf-zero} in the case $u_1=u_2$ (estimate \eqref{int_maggiorato} ahead  in the case $u_1\neq u_2$, respectively), that $\sup_n  \sup_{t\in [t_1^n,t_2^n]} \ene t{\teta_n(t)} \leq C$. Then,
\eqref{in-a-sublevel} ensues. 
\par\noindent
\textbf{\emph{Ad Claim (1)}:}
By contradiction,
suppose that $u_1\neq u_2$. 
Thanks to \eqref{strong-critical}
the set $\subl\rho \cap \calC(t)$ is finite (since the energy sublevel $\subl\rho$ is compact  in $\xfin $ by \eqref{coercivita}),
hence  there exists $\overline\delta=\overline\delta(t,u_1,u_2)$ such that
\begin{equation}
\label{delf-hat-delta}
B(x,2\delta)\cap B(y,2\delta)=\emptyset,
\ \ \mbox{for every }x,y\in(\mathcal {C}(t){\cap} \subl\rho)\cup\{u_1,u_2\}\mbox{ with }x\neq y \qquad \text{ for all } 0<\delta\leq \overline{\delta}.
\end{equation}
Observe that $u_1$ may well belong to $\mathcal C(t)$ as well as not, and the same for $u_2$.
Let us introduce the compact set $K_\delta$ defined by
\begin{equation}
\label{set-Kdelta}
K_\delta:=\subl\rho\setminus\bigcup_{x\in (\mathcal {C}(t){\cap} \subl\rho) \cup\{u_1,u_2\}}B(x,\delta)
\end{equation}
and remark that $\min_{u\in K_\delta}\| \rmD\mathcal E_t(u)\|>0.$
It follows from Lemma \ref{lemma:auxi-v}  that for some $\alpha=\alpha(t,u_1,u_2)>0$
\begin{equation}
\label{quoted-later} 
e_\delta:=\min_{u\in
K_\delta,r\in[t-\alpha,t+\alpha]}\|\rmD\mathcal E_r(u)\|>0.
\end{equation}
Note that $[t_1^n,t_2^n]\subset[t-\alpha,t+\alpha]$ for every $n$ sufficiently large.
Moreover, from \eqref{new-cost}
and from the definition of $K_\delta$ we obtain that 
$\{r\in[t_1^n,t_2^n]\,:\,\teta_n(r)\in K_\delta\}\neq\emptyset$ for every $n$ large enough, and that
$\teta_n(r_1)\in\pa B(u_1,\delta)$, $\teta_n(r_2)\in\pa B(u_2,\delta)$, for some
$r_1$, $r_2\in\{r\in[t_1^n,t_2^n]\,:\,\teta_n(r)\in K_\delta\}$ with $r_1\neq r_2$.
Thus, by \eqref{quoted-later},
\begin{eqnarray}\label{ineq:lemma1}
\int_{t_1^n}^{t_2^n}\|\rmD\mathcal E_r(\teta_n(r))\| 
\|\teta_n'(r)\|dr\!\!\!&\geq&\!\!\!
 \int\limits_{\{r\in[t_1^n,t_2^n]\,:\,\teta_n(r)\in K_\delta\}}
 \|\rmD\mathcal E_r(\teta_n(r))\| \|\teta_n'(r)\|dr\nonumber\\
\!\!\!&\geq&\!\!\!
                e_\delta\int\limits_{\{r\in[t_1^n,t_2^n]\,:\,\teta_n(r)\in K_\delta\}}\|\teta_n'(r)\|dr,\nonumber\\
\!\!\!&\geq&\!\!\!e_\delta \min_{x,y\in(\mathcal C(t){\cap}K_\delta)\cup\{u_1,u_2\}}(\|x-y\|-2\delta) \doteq \eta.
\end{eqnarray}
Observe that $\eta$ is positive in view of (\ref{quoted-later}) and of the definition of $\delta$
from \eqref{delf-hat-delta}. Thus we have a contradiction with \eqref{liminf-zero}.

\par\noindent
\textbf{\emph{Ad Claim (2)}:} 
Suppose that $u_1 \neq u_2$. 
Up to a subsequence we can suppose that there exists
\begin{equation*}
\lim_{n \to \infty} \int_{t_1^n}^{t_2^n}
\minpartial{\cE}{r}{\teta_n(r)} \, \mdt{\teta}{n}{r}  \dd r =:L_t>0,
\end{equation*}
hence
\begin{equation}\label{int_maggiorato}
\sup_{n \in \N} \int_{t_1^n}^{t_2^n}
 \minpartial{\cE}{r}{\teta_n(r)} \, \mdt{\teta}{n}{r} \dd r \leq C <\infty.
\end{equation}
We split the proof of \eqref{liminf-lemmone} in several steps.
\par
\noindent \textbf{Step $1$: reparameterization.} Let us define, for
every $r\in[t_1^n,t_2^n]$,
\begin{equation*}
s_n(r):=r+\int_{t_1^n}^r \minpartial{\cE}{\tau}{\teta_n(\tau)} \,
\mdt{\teta}{n}{ \tau} \dd \tau.
\end{equation*}
Also, we set
\begin{equation*}
s_1^n:=s_n(t_1^n)=t_1^n,\qquad\qquad s_2^n:=s_n(t_2^n),
\end{equation*}
and note that
\begin{equation*}
s_1^n\to t,\qquad\qquad s_2^n\to(t+L_t)>t.
\end{equation*}
Since $s_n'>0$, we define
\begin{equation*}
r_n(s):=s_n^{-1}(s)\ \ \mbox{and}\ \
\tilde\teta_n(s):=\teta_n(r_n(s))\ \ \mbox{for every}\ \
s\in[s_1^n,s_2^n].
\end{equation*}
Observe that
\begin{equation}\label{tilde_teta_x1x2}
\tilde\teta_n(s_1^n)\to u_1,\qquad\qquad\tilde\teta_n(s_2^n)\to u_2,
\end{equation}
that $(r_n)_n$ is equi-Lipschitz and that
\begin{equation}\label{super_lem_01}
 \minpartial{\cE}{r_n(s)}{\tilde\teta_n(s)} \, \mdt{\tilde\teta}{n}{
s} = 1-\frac1{1+  \minpartial{\cE}{r_n(s)}{\tilde\teta_n(s)} \,
\mdt{\tilde\teta}{n}{ s} }\leq1 \ \ \foraa\,s\in(s_1^n,s_2^n).
\end{equation}
The change of variable formula yields
\begin{equation}\label{super_lem_1}
\int_{t_1^n}^{t_2^n} \minpartial{\cE}{r}{\teta_n(r)} \,
\mdt{\teta}{n}{r}  \dd r=\int_{s_1^n}^{s_2^n}
\minpartial{\cE}{r_n(s)}{\tilde\teta_n(s)} \, \mdt{\tilde\teta}{n}{
s} \dd s.
\end{equation}

\noindent \textbf{Step $2$: localization and equicontinuity
estimates.} 
Let $\overline\delta,\, \delta>0$, $K_\delta$, and $e_\delta$ be as in 
\eqref{delf-hat-delta}, \eqref{set-Kdelta}, \eqref{quoted-later}, and define
%
the open set
\begin{equation*}
A_n^{\delta}:=\left\{s\in(s_1^n,s_2^n)\,:\,\tilde\teta_n(s)\in\mathrm{int}
(K_{\delta}) \right\}.
\end{equation*}
Observe that $A_n^{\delta}\neq\emptyset$ for every $n$ sufficiently
large, in view of the definition of $K_{\delta}$ and of
(\ref{tilde_teta_x1x2}).
We write $A_n^{\delta}$ as the countable union of its connected
components
\begin{equation}\label{super_lem_2}
A_n^{\delta}=\bigcup_{k=1}^{\infty}(a^{\delta}_{n,k},b^{\delta}_{n,k})
\quad \text{with }\ b^{\delta}_{n,k}\leq a^{\delta}_{n,k+1}\ \text{
for all }\ k\in \N.
\end{equation}
Inequality \eqref{super_lem_01}, the definition   \eqref{quoted-later}
of $e_\delta$, and the definition of $a_{n,k}^{\delta}$ and
$b_{n,k}^{\delta}$ imply that
\begin{equation}\label{tilde_teta_equilip}
e_{\delta} \mdt{\tilde\teta}{n}s \leq
\minpartial{\cE}{r_n(s)}{\tilde\teta_n(s)}\, \mdt{\tilde\teta}{n}s
\leq1\quad\foraa\,s\in\left(a_{n,k}^{\delta},b_{n,k}^{\delta}\right).
\end{equation}
Furthermore, it is clear that
\begin{equation}\label{tilde_teta_bordo}
\tilde\teta_n(a^{\delta}_{n,k})\in\pa B(x,\delta),\ \
\tilde\teta_n(b^{\delta}_{n,k})\in\pa B(y
,\delta)\quad\mbox{for some}\ \ x, \, y\in(\mathcal
C(t){\cap} \subl\rho)\cup\{u_1,u_2\}.
\end{equation}
Note that it may happen $x=y$.
 Nonetheless, from now on we will just focus on the case where
$ x\neq y$  in \eqref{tilde_teta_bordo}  and we will
show  that there is a \emph{finite} number of intervals
$(a^{\delta}_{n,k},b^{\delta}_{n,k})$ on which $\tilde\teta_n$ travels  
from one ball to another, centered at a different point in
$(\mathcal C(t){\cap} \subl\rho)\cup\{u_1,u_2\}$. In this way, we will conclude that
the function $\teta$ of the statement consists of a finite number of
$\mathrm{C}_{\mathrm{loc}}^{\mathrm{lip}}$-pieces. To this aim, let
us introduce the set
\begin{equation}
\label{def-b-enne-delta}
\begin{gathered}
B_n^{\delta}:=\bigcup_{(a^{\delta}_{n,k},b^{\delta}_{n,k})
\in\mathfrak{B}_n^{\delta}} (a^{\delta}_{n,k},b^{\delta}_{n,k})
\quad \text{with}
\\
 \mathfrak{B}_n^{\delta}=
\left\{(a^{\delta}_{n,k},b^{\delta}_{n,k})\subset
A_n^{\delta}\,:\, \tilde\teta_n(a^{\delta}_{n,k})\in\pa B(x
,\delta),\ \tilde\teta_n(b^{\delta}_{n,k})\in\pa B(y,\delta)\text{ for } x, \, y\in\mathcal
C(t){\cap} \subl\rho, \   x \neq y\right\}.
\end{gathered}
\end{equation}
From \eqref{super_lem_1}, \eqref{tilde_teta_equilip}, and the
definition of $A_n^{\delta}$ and $B_n^{\delta}$, we
obtain
\begin{eqnarray}
\int_{t_1^n}^{t_2^n} \minpartial{\cE}{r}{\teta_n(r)} \,
\mdt{\teta}{n}{r} \dd r
   &\geq&\int_{A_n^{\delta}}
   \minpartial{\cE}{r_n(s)}{\tilde\teta_n(s)} \,  \mdt{\tilde\teta}{n}{s}
\dd
   s\nonumber\\
  &\geq& e_{\delta}\int_{B_n^{\delta}}
   \mdt{\tilde\teta}{n}{s} \dd
  s\nonumber\\
  &\geq& e_{\delta}\sum_{(a^{\delta}_{n,k},b^{\delta}_{n,k}) \in  \mathfrak{B}_n^{\delta}}\overline m,
\label{bar-m}
\end{eqnarray}
where
$0<\overline m:= \min_{\stackrel{x,y\in\mathcal
C(t)\cup\{u_1,u_2\}}{\pij xt \neq \pij y t}}(\|x-y\|-2\overline\delta)$.
Inequality \eqref{bar-m}, together with estimate
\eqref{int_maggiorato}, implies that $B_n^{\delta}$ has a finite
number $N(n,\delta)$ of components, more precisely
\begin{equation}\label{stima_N_n_d}
N(n,\delta)\leq\frac C{e_{\delta}\overline m}\ \mbox{ for every}\
0<\delta\leq\overline\delta,\ n\in\N,
\end{equation}
with $C$ from \eqref{int_maggiorato}. In what follows, we will show
that  we may take $ N(n,\delta)$ to be bounded uniformly w.r.t.\  $n
\in \N$ and $\delta>0$ (cf.\ \eqref{very-strong-below} ahead). 

\par
For this, we need to fix some preliminary remarks. In view of the
ordering assumed in \eqref{super_lem_2}, we have that
\begin{equation}\label{re-order}
\tilde{\teta}_n(a^{\delta}_{n,1}) \in \partial B(u_1
,\delta), \qquad \tilde{\teta}_n(b^{\delta}_{n,N(n,\delta)}) \in
\partial B(u_2,\delta).
\end{equation}
Also, observe that, up to throwing some of the intervals
$(a^{\delta}_{n,k},b^{\delta}_{n,k})\in\mathfrak{B}_n^{\delta}$
away, we may suppose that for every fixed $x\in
(\mathcal{C}(t){\cap} \subl\rho)\cup\{u_1,u_2\}$, if
$\tilde\teta(a^{\delta}_{n,k})\in\partial B(x,\delta)$ for
some $k$, then $\tilde\teta(b^{\delta}_{n,m})\notin\partial B(x,\delta)$ for every $m>k$. Finally, note that for all
$(a^{\delta}_{n,k}, b^{\delta}_{n,k}) \in \mathfrak{B}_n^{\delta}$
there holds
\begin{equation}
\label{e:btw} \frac1{e_\delta} ( b^{\delta}_{n,k} - a^{\delta}_{n,k}
) \geq \|\tilde{\teta}_n(b^{\delta}_{n,k}) -
\tilde{\teta}_n(a^{\delta}_{n,k})   \|  \geq \overline m,
\end{equation}
where the first inequality ensues from \eqref{tilde_teta_equilip},
 and the second one is due to the definition of $\overline m$.
\begin{remark}
\upshape \label{rem:v1} Observe that a bound for  $ N(n,\delta)$,
uniform with respect to  $n \in \N$ \emph{and} $\delta>0$, cannot be
directly deduced from \eqref{stima_N_n_d} since the constant
$C/(e_{\delta}\overline m)$ grows as $\delta$ decreases. Indeed,
$e_{\delta}$ goes to zero as $\delta\to0$.
\end{remark}

\noindent \textbf{Step $3$: compactness.} We will prove the
following
\\
\emph{\underline{Claim}:  there exist a sequence
 $(n_j,\delta_{m_j})_j$, such that
\begin{equation}
\label{very-strong-below} N(n_j,\delta_{m_j})=N \qquad\mbox{for
every }j \in \N,
\end{equation}
 a partition
\begin{equation}\label{partizione}
\{
t\leq\alpha_1<\beta_1\leq\alpha_2<\beta_2\leq...\leq\alpha_N<\beta_N\leq
t+L_t  \}\quad \text{of}\quad[t,t+L_t], \ \text{ and }
\end{equation}
a curve 
$\teta\in\mathrm{C}_{\mathrm{loc}}^{\mathrm{lip}}
\left(\bigcup_{k=1}^N(\alpha_k,\beta_k);\Hilbert\right)$
such that, in the limit $j\to\infty$,
\begin{equation}
\label{tipo_di_convergenza}
\begin{aligned}
& \tilde\teta_{n_j}\to\teta\quad\mbox{uniformly on compact subsets
of }
                   \bigcup_{k=1}^N(\alpha_k,\beta_k),
                   \\
   &
                   \tilde\teta'_{n_j} \weaksto \teta'  \quad\mbox{in $L^\infty (\alpha_k+\eta,\beta_k-\eta;\Hilbert)$
                   \quad for every $\eta>0$ and $k=1,\ldots,N$.}
 \end{aligned}
\end{equation}
Therefore,
$\teta(s) \in \subl\rho$ for every $s\in \bigcup_{k=1}^N(\alpha_k,\beta_k)$.
} 
\par
First of all, let us observe that, since $(N(n,\delta))_n$ is a
bounded sequence by \eqref{stima_N_n_d}, there exists a subsequence
$(n_l^{\delta})_l$ and an integer $N(\delta)$ such that
\begin{equation}
\label{citata-piu-sotto}
  N(n_l^{\delta},\delta)\to N(\delta)\quad\mbox{as }l\to\infty.
\end{equation}
  Clearly, since $u_1 \neq u_2$, taking \eqref{tilde_teta_x1x2} into account we see that $N(\delta)\geq1$
for every $0<\delta\leq\overline\delta$.
 Also, for every fixed $n\in \N$, we have that
\begin{equation}
\label{N_decresce}
N(n,\delta)\ \mbox{ decreases as } \delta \mbox{ decreases}.
\end{equation}
Indeed, if $\delta_1>\delta_2$, then $K_{\delta_1}\subset
K_{\delta_2}$ and in turn $B_n^{\delta_1}\subset B_n^{\delta_2}$.
This means that, for every $k\in\{1,...,N(n,\delta_1)\}$, we have 
\begin{equation}
\label{piccina}
(a_{n,k}^{\delta_1},b_{n,k}^{\delta_1})\subset
(a_{n,j_k}^{\delta_2},b_{n,j_k}^{\delta_2}),  
\qquad\mbox{for some}\quad j_k\in\{1,...,N(n,\delta_2)\}. 
\end{equation}
At the same time,
$(a_{n,k}^{\delta},b_{n,k}^{\delta})\subset[s_1^n,s_2^n]$ for
every $0<\delta\leq\overline\delta$, for every $n$, and for every
$k=1,...,N(n,\delta)$, and there holds $[s_1^n,s_2^n]\to[t,t+L_t]$ as
$n\to\infty$. Therefore we may suppose that there exists $\eta>0$
such that for every $n \in \N $ and $k=1,...,N(n,\delta)$ there
holds $(a_{n,k}^{\delta},b_{n,k}^{\delta})\subset [t-\eta,t+L_t+\eta]$.
This fact, coupled with \eqref{piccina}, gives \eqref{N_decresce}.

In order to prove \eqref{very-strong-below}, we develop the
following diagonal argument. Consider a sequence
$(\delta_m)_m\subset(0,\overline\delta]$ such that $\delta_m\to0$,
as $m\to\infty$. Using \eqref{citata-piu-sotto} and
\eqref{N_decresce}, it is possible to construct for each $m\in \N$ a
subsequence $(n_l^m)_l$, where $n_l^m$ is a short-hand notation for
$n_l^{\delta_m}$, such that
\begin{equation}\label{diag_1}
N(n_l^m,\delta_m)\to N(\delta_m)\qquad\mbox{as}\qquad
l\to\infty\quad\mbox{ for every }m \in \N,
\end{equation}
\begin{equation}\label{diag_2}
N(\delta_1)\geq N(\delta_2)\geq...\geq N(\delta_m)\geq...\geq 1,
\end{equation}
and
\begin{equation}\label{diag22}
a_{n_l^m,k}^{\delta_m}\to \alpha_k^m,\quad
b_{n_l^m,k}^{\delta_m}\to\beta_k^{m}\quad \mbox{as
}l\to\infty\quad\mbox{for all }k=1,...,N(\delta_m)  \mbox{ and  all
}m \in \N.
\end{equation}
Due to \eqref{e:btw}, one has $\alpha_k^m< \beta_k^{m}$ for all
$k=1,...,N(\delta_m)$ and all $m \in \N$. Now, since
$(N(\delta_m))_m$ is a decreasing sequence of integers, we may
suppose that
\begin{equation}\label{diag_3}
\mbox{there exists }\ N\in \N\mbox{ such that }\ N(\delta_m)=N\
\mbox{ for every }m \in \N,
\end{equation}
and that there exist the limits
\begin{equation}\label{diag33}
\alpha_k:=\lim_{m\to\infty}\alpha_k^m,\qquad\beta_k:=\lim_{m\to\infty}\beta_k^m
\qquad\mbox{for every }k=1,...,N.
\end{equation}
Note that the points $\alpha^k, \, \beta^k$ satisfy
\eqref{partizione} with $N$ from \eqref{diag_3}.
Hence, choose $k\in\{1,...,N\}$ and observe that for every $j\in\N$
arbitrarily large there exists $m_j$ and $l_j$ such that
\begin{equation}\label{interv_fisso}
\left[\alpha_k+\frac1j,\beta_k-\frac1j\right]\subset
\left(a_{n_l^{m_j},k}^{\delta_{m_j}},b_{n_l^{m_j},k}^{\delta_{m_j}}\right)\qquad\text{for
every $l\geq l_j$,}
\end{equation}
and
\begin{equation}\label{super-lemma:nuova}
\left|a_{n_l^{m_j},k}^{\delta_{m_j}}-\alpha_k^{m_j}\right|+
\left|b_{n_l^{m_j},k}^{\delta_{m_j}}-\beta_k^{m_j}\right|\leq\frac1{m_j}\qquad\text{for
every $l\geq l_j$.}
\end{equation}
Moreover, we can suppose that $m_j$, $l_j\to\infty$, as
$j\to\infty$. We combine \eqref{interv_fisso} with  estimate
\eqref{tilde_teta_equilip} for $(\tilde{\teta}'_n)_n$
 and the fact that $\tilde{\teta}_n(s) \in \subl\rho$ for every $s\in [s_1^n,s_2^n]$. The
Arzel\`a-Ascoli Theorem
ensures  that, up to a subsequence,
\begin{subequations}
\label{cvg_sotto}
\begin{align}\label{cvg_sotto-1}
&
\begin{aligned}
& \tilde\teta_{n_{l}^{m_j}}  \to\teta\quad\mbox{in}\quad
\rmC^0\left(
\left[\alpha_k+\frac1j,\beta_k-\frac1j\right];\Hilbert\right) \quad
\text{as } l \to \infty,
\end{aligned}
\\
& \label{cvg_sotto-2}
\begin{aligned}
 \tilde\teta'_{n_{l}^{m_j}}  \weaksto \teta' \quad\mbox{in }
L^\infty \left(\alpha_k+\frac1j,\beta_k-\frac1j;\Hilbert\right)\quad
\text{as } l \to \infty
\end{aligned}
\end{align}
\end{subequations}
for some
$\teta\in\mathrm{C}^{\mathrm{lip}}\left(\left[\alpha_k+\frac1j,\beta_k-\frac1j\right];\Hilbert\right)$.

If at each step $j$ we extract a subsequence from the previous one,
we may obtain a sequence $(n_{l_j}^{m_j})_j$, which we relabel by
$(n_j)_j$, and a unique
$\teta\in\mathrm{C}_{\mathrm{loc}}^{\mathrm{lip}}\left(\bigcup_{k=1}^N(\alpha_k,\beta_k);\Hilbert\right)$,
such that for all $ j \in \N $ and $k\in\{1,...,N\}$ there holds
\begin{align}\label{diag_4}
&
\left[\alpha_k+\frac1j,\beta_k-\frac1j\right]\subset \left(\tilde
a_{j,k},\tilde b_{j,k}\right),\quad\mbox{where}\quad \tilde
a_{j,k}:=a_{n_j,k}^{\delta_{m_j}},\ \ \tilde
b_{j,k}:=b_{n_j,k}^{\delta_{m_j}},
\\
& \label{convergenza_j} \|\tilde\teta_{n_j}(s)-
\teta(s)\|<\frac1j\quad\mbox{for every }s\in
\left[\alpha_k+\frac1j,\beta_k-\frac1j\right]\,.
\end{align}
Therefore, we have  proved
\eqref{very-strong-below}--\eqref{tipo_di_convergenza}. From
(\ref{diag33}) and (\ref{super-lemma:nuova}) we obtain also that
\begin{equation}\label{super_lem:nuova2}
\tilde a_{j,k}\to\alpha_k,\qquad\tilde
b_{j,k}\to\beta_k\quad\mbox{as }j\to\infty,
\end{equation}
where $\tilde a_{j,k}$, $\tilde b_{j,k}$ are defined in
(\ref{diag_4}). These observations will be useful in Step 5.

\begin{remark}
\upshape \label{rmk:v2} Let us recall   (cf.\
\eqref{very-strong-below}), that $N$ is the number of the pieces of
the trajectory of $\tilde\teta_{n_j}$ which go from $\pa B(x
,\delta_{m_j})$ to $\pa B(y ,\delta_{m_j})$, for some $x, \,
y\in (\mathcal {C}(t){\cap} \subl\rho)\cup\{u_1,u_2\}$ with $x\neq y$. Thus, we have so
far excluded that, for example, on some interval $(\tilde
a_{j,k},\tilde b_{j,k})$ the trajectory of $\tilde\teta_{n_j}$ runs
from $\pa B(x,\delta_{m_j})$ to $\pa B(x,\delta_{m_j})$. 
Moreover, so far we have overlooked what happens to the trajectory
of $\tilde\teta_{n_j}$ on the interval $[\tilde b_{j,k},\tilde
a_{j,k+1}]$. It is not difficult to imagine that, if
$\beta_k<\alpha_{k+1}$ some
``loops'' around a certain 
connected component of $(\mathcal {C}(t){\cap} \subl\rho)\cup\{u_1,u_2\}$  may have been
created by the trajectories of $\tilde\teta_{n_j}$ on $[\tilde
b_{j,k},\tilde a_{j,k+1}]$ as $j\to\infty$. Note that we cannot
deduce that the number of these loops is definitely bounded, as we
have done for $N(n_j,\delta_{m_j})$.
\end{remark}

 \noindent
 \textbf{Step $4$: passage to the limit.}
In order to take the limit of the integral term in
\eqref{liminf-lemmone}, we observe that
\begin{equation}
\label{catenone}
\begin{aligned}
\int_{t_1^{n_j}}^{t_2^{n_j}} \minpartial{\cE}{r}{\teta_{n_j}(r)}\,
\mdt{\teta}{n_j}{r} \dd r
 & \geq\sum_{k=1}^N\int_{\tilde a_{j,k}}^{\tilde b_{j,k}}
  \minpartial{\cE}{r_{n_j}(s)}{\tilde\teta_{n_j}(s)}\, \mdt{\tilde\teta}{n_j}{s} \dd
  s
    \\
     & \geq
    \sum_{k=1}^N\int_{\alpha_k +1/j}^{\beta_k -1/j}
 \minpartial{\cE}{r_{n_j}(s)}{\tilde\teta_{n_j}(s)}\, \mdt{\tilde\teta}{n_j}{s} \dd s,
    \end{aligned}
\end{equation}
where we have used  \eqref{very-strong-below}  and \eqref{diag_4}.
We now pass to the limit as $j \to \infty$ in \eqref{catenone}.
Observe that, since $(r_{n_j}(s))_j\subset[t_1^{n_j},t_2^{n_j}]$
for every $s\in[s_1^{n_j},s_2^{n_j}]$, then $r_{n_j}(s)\to t$ as
$j\to\infty$. 
Hence,  
the first convergence in \eqref{tipo_di_convergenza}
yields
\begin{equation}\label{liminf-ioffe}\begin{aligned}
 \lim_{j \to \infty} \minpartial{\cE}{r_{n_j}(s)}{\tilde\teta_{n_j}(s)}   = \minpartial{\cE} t{\teta(s)}  \text{ for every } s \in
[\alpha_k+\eta,\beta_k-\eta], \ \eta>0,   \ k=1,\ldots,
N.\end{aligned}
\end{equation}
Combining \eqref{liminf-ioffe} with the second of  \eqref{tipo_di_convergenza} and applying
Ioffe's Theorem \cite{Ioffe},
we have that
\begin{equation}
\label{eq:towards-integr} \liminf_{j\to\infty} \int_{\alpha_k
+1/j}^{\beta_k -1/j}
  \minpartial{\cE}{r_{n_j}(s)}{\tilde\teta_{n_j}(s)}\, \mdt{\tilde\teta}{n_j}{s} \dd
  s   \geq  \int_{\alpha_k +\eta}^{\beta_k -\eta}  \minpartial \calE t{\teta(s)} \, \|\teta'(
  s)\| \dd s
\end{equation}
   for all $ \eta>0, \ k=1,\ldots, N.$
From \eqref{eq:towards-integr} and \eqref{int_maggiorato}
 it follows that the map $s \mapsto  \minpartial \calE t{\teta(s)} \, \|\teta'( s)\| $ is integrable on $(\alpha_k,\beta_k) $ for all $k=1,\ldots, N$. Summing up,
we conclude that
\begin{equation}
\begin{aligned}\label{liminf-quasi}
 & \liminf_{j\to\infty}\int_{t_1^{n_j}}^{t_2^{n_j}}
 \minpartial{\cE}{r}{\teta_{n_j}(r)}\, \mdt{\teta}{n_j}{r} \dd r
 \geq\sum_{k=1}^N\int_{\alpha_k}^{\beta_k}
 \minpartial \calE t{\teta(s)} \, \|\teta'( s)\| \dd
 s.
\end{aligned} \end{equation}

\noindent \textbf{Step $5$: conclusion of the proof of Proposition \ref{super-lemmone}.} Relying on the 
previously proved part (1) of the statement, the first of (\ref{tipo_di_convergenza}),
and the inclusion in \eqref{diag_4}, we will now show that
\begin{equation}\label{lim_teta_1}
\lim_{s\to\alpha_1^+} \teta(s) =u_1
,\qquad\lim_{s\to\beta_N^-} \teta(s) =u_2
\end{equation}
and that
\begin{equation}\label{lim_teta_2}
\lim_{s\to\beta_k^-} \teta(s) =\lim_{s\to\alpha_{k+1}^+}\teta(s) = x  \quad\mbox{ for
some }x\in(\mathcal C(t){\cap} \subl\rho)\cup\{u_1,u_2\},
\end{equation}
for every $k=1,...,N-1$.
Let us only check the first limit in \eqref{lim_teta_1}, since the
other limits can be verified in a similar way.
Let $(s_i)_i\subset(\alpha_1,\beta_1)$ be a sequence such that
$s_i\to\alpha_1^+$ as $i \to \infty$. We want to prove that
\begin{equation}\label{nuova_da_dim}
\lim_{i\to\infty}\teta(s_i) = u_1.
\end{equation}
Now, let us fix $i \in \N$: the first of (\ref{tipo_di_convergenza})
gives that
$\tilde\teta_{n_j}(s_i)\to\teta(s_i)$ as $j\to\infty$. 
 In particular,
there exists a strictly increasing sequence $(j_i)_i$ such that
\begin{equation}\label{lemmone:nuova22}
\|\tilde\teta_{n_{j_i}}(s_i) -  \teta(s_i)\| \leq\frac1i
\qquad\mbox{for every }i \in \N.
\end{equation}
Note that $\tilde a_{j_i,1}\to\alpha_1$ as $i\to\infty$, in view of
(\ref{super_lem:nuova2}). Moreover, from the definition
(\ref{diag_4}) of $\tilde a_{j_i,1}$ and \eqref{re-order} it follows
that
\begin{equation}\label{lemmone:nuova25}
\lim_{i\to\infty}\tilde\teta_{n_{j_i}}(\tilde a_{j_i,1}) = u_1\,.
\end{equation}
Next, observe that from (\ref{super_lem_01}) and from the fact that
$s_i$, $\tilde a_{j_i,1}\to\alpha_1$ as $i\to\infty$, we have that
\begin{equation}
\label{convergences-integrals} \int_{s_i}^{\tilde a_{j_i,1}}
\minpartial{\cE}{r_{n_{j_i}}(s)}{\tilde\teta_{n_{j_i}}(s)} \, \mdt
{\tilde \teta}{n_{j_i}}{s}\dd s \leq|s_i-\tilde
a_{j_i,1}|\to0\quad\mbox{as }i\to\infty.
\end{equation}
Also, we have that
\begin{equation}
\label{citata-dopo-oh-yes}
\int_{s_i}^{\tilde a_{j_i,1}}
 \minpartial{\cE}{r_{n_{j_i}}(s)}{\tilde\teta_{n_{j_i}}(s)} \, \mdt {\tilde \teta}{n_{j_i}}{s}\dd s
 =
\int_{r_i}^{\tilde r_i}
\minpartial{\cE}{r}{\teta_{n_{j_i}}(r)}\,
\mdt {\teta}{n_{j_i}}{r}\dd r,
\end{equation}
for some $(r_i)_i$, $(\tilde
r_i)_i\subset\left[t_1^{n_{j_i}},t_2^{n_{j_i}}\right]$, where
\begin{equation}\label{lemmmone:nuova3}
\teta_{n_{j_i}}(r_i)=\tilde\teta_{n_{j_i}}(s_i),\quad
\teta_{n_{j_i}}(\tilde r_i)=\tilde\teta_{n_{j_i}}(\tilde
a_{j_i,1})\qquad\mbox{for every }i \in \N.
\end{equation}
Furthermore, we can suppose that, up to a subsequence,
$r_i\leq\tilde r_i$ for every $i$, and that
\begin{equation}
\label{convergences-1-hat} \teta_{n_{j_i}}(r_i)\to\hat
x\qquad\mbox{for some }\hat x\in \xfin.
\end{equation}
We combine
this fact with the limit
\begin{equation*}
\teta_{n_{j_i}}(\tilde r_i)\to u_1,\qquad\mbox{as }i\to\infty
\end{equation*}
which comes from (\ref{lemmone:nuova25}) and the second of (\ref{lemmmone:nuova3}),
and apply Proposition \ref{super-lemmone} (1) to the sequence $(\teta_{n_{j_i}})_i$ on the shrinking interval $[r_i,\tilde{r}_i]$, using that 
$\int_{r_i}^{\tilde r_i}
\minpartial{\cE}{r}{\teta_{n_{j_i}}(r)}\,
\mdt {\teta}{n_{j_i}}{r}\dd r \to 0$ by \eqref{citata-dopo-oh-yes}. 
Therefore,
\begin{equation}\label{lemmone:nuova4}
\tilde\teta_{n_{j_i}}(s_i)=\teta_{n_{j_i}}(r_i)\to u_1\qquad\mbox{as }i\to\infty,
\end{equation}
Inequality (\ref{lemmone:nuova22}) and convergence (\ref{lemmone:nuova4})
imply (\ref{nuova_da_dim}).

By the limits in \eqref{lim_teta_1} and \eqref{lim_teta_2} we can
trivially extend $\teta$ to the whole interval $[\alpha_1,\beta_N]$
and obtain, from (\ref{liminf-quasi}), that
\begin{equation*}
\liminf_{n \to \infty} \int_{t_1^n}^{t_2^n}
\minpartial{\cE}{r}{\teta_n(r)} \, \|\teta'_{n}(r)\| \dd r \geq
\int_{\alpha_1}^{\beta_N} \minpartial {\cE}{t}{\teta(s)} \,
\|\teta'( s)\|
 \dd s.
\end{equation*}
Thus, we have deduced the $\liminf$-inequality 
\eqref{liminf-lemmone},  with a curve defined
on the interval $[a,b]=[\alpha_1,\beta_N]$.

Finally, by  the scaling
invariance of the integral on the right-hand side of
\eqref{liminf-lemmone}, we can reparameterize the limiting curve
$\teta$ in such a way as to have it defined on the interval $[0,1]$,
in accord with the definition \eqref{def-admissible-class-simpler} of
admissible curves.
 This concludes the proof.
\QED

We now give a variant of Proposition \ref{super-lemmone}, in which the curves $\teta_n$
belong to the class $\admis{u_1^n}{u_2^n}{t}$, 
for $t\in [0,T]$ fixed, and   with
$(u_1^n), \, (u_2^n)\subset \xfin$ given sequences, and the
integrands $\minpartial{\cE}{r}{\teta_n(r)}\, \mdt{\teta} n r $  in
\eqref{liminf-lemmone} are replaced by
$\minpartial{\cE}{t}{\teta_n(r)}\, \mdt{\teta} n r $. 
\begin{proposition}
\label{lemma:extension-lemmone}
Assume \eqref{Ezero}--\eqref{strong-critical}.
Given $t \in [0,T]$, $\rho>0$,
and $u_1$, $u_2\in\domainenergy$,
let $(u_1^n)_n,\, (u_2^n)_n $ fulfill 
$u_1^n\to u_1, $ $u_2^n\to u_2$, 
and let $\teta_n\in\admis{u_1^n}{u_2^n}{t}$ for every $n$.
The following two implications hold:
\begin{compactenum}
\item
If
$
\liminf_{n \to \infty}\int_0^1
\minpartial{\cE}{t}{\teta_n(s)}\,\mdt{\teta} n s \dd s=0
$,
then $u_1=u_2$;
\item
If $u_1\neq u_2$, so that
\begin{equation*}
\liminf_{n \to \infty}\int_0^1
\minpartial{\cE}{t}{\teta_n(s)}\,\mdt{\teta} n s \dd s>0,
\end{equation*}
then there exists  $\teta \in\admis{u_1}{u_2}{t}$
such that
\begin{equation}
\label{liminf-cost}
\liminf_{n \to \infty}
\int_0^1 \minpartial{\cE}{t}{\teta_n(s)}\,\mdt{\teta} n s  \dd s
\,\geq\,
\int_0^1 \minpartial {\cE}{t}{\teta(s)}\,
\|\teta' (s)\| \dd s.
\end{equation}
\end{compactenum}
\end{proposition}
\begin{proof}
We will only  sketch the proof, 
dwelling on the differences with the argument for  Proposition
\ref{super-lemmone}.

\par
By the very same arguments developed at the beginning of Prop.\ 
\ref{super-lemmone}, we conclude that the (images of)  all  the  curves 
$\teta_n$ in fact lie in some energy sublevel. 
\par The proof of 
\emph{\textbf{Claim (1)}} follows the very same lines as for Prop.\   
\ref{super-lemmone}.

\emph{\textbf{Ad Claim (2):}}
By definition of $\admis{u_1^n}{u_2^n}{t,\tilde\rho}$, we have that
there exists a partition $0=\tau_0^n<\tau_1^n<...<\tau_{M_n}^n=1$
such that $\teta_n(0) =u_1^n$, $\teta_n(1) =u_2^n$,
$\teta_n|_{(\tau_i^n,\tau_{i+1}^n)} \in
\mathrm{C}_{\mathrm{loc}}^{\mathrm{lip}}
((\tau_i^n,\tau_{i+1}^n);\Hilbert) $ for all $ i=0,\ldots, M_n-1,$
and the curves $\teta_n|_{(\tau_i^n,\tau_{i+1}^n)}$ connect
different connected components of
$\mathcal{C}(t) \cap \subl{\tilde\rho}$.
In analogy with the proof of Prop.\ \ref{super-lemmone}, we use
the rescaling $r_n$ defined as the inverse of the function
$
s_n(r):=\int_0^r \minpartial{\cE}{t}{\teta_n(s)} \, \mdt \teta n s
\dd s,$
for $r\in[a,b],
$
and set $\tilde\teta_n(s):=\teta_n(r_n(s))$ for every
$s\in\left[\tilde a_n,\tilde b_n\right]$, where $\tilde
a_n:=r_n^{-1}(0)$ and $\tilde b_n:=r_n^{-1}(1)$. Then,
 there exists a partition
$ \tilde a_n=\sigma_0^n<\sigma_1^n<...<\sigma_{M_n}^n=\tilde b_n, $
 with
$\tilde\teta_n(\tilde a_n) =u_1^n$, $\tilde\teta_n(\tilde b_n)
=u_2^n$, such that $\tilde\teta_n|_{(\sigma_i^n,\sigma_{i+1}^n)}\in
\mathrm{C}_{\mathrm{loc}}^{\mathrm{lip}}
((\sigma_i^n,\sigma_{i+1}^n);\Hilbert) $ for all $i=0,\ldots, M_n-1$.
Moreover,
\begin{equation*}
\minpartial{\cE}{t}{\tilde\teta_n(s)} \, \mdt{\tilde \teta}n s \leq
1
 \quad\foraa\ s\in(\tilde a_n,\tilde b_n),
\end{equation*}
and
\begin{equation*}
\int_0^1
 \minpartial{\cE}{t}{\teta_n(r)}\, \mdt{\teta} n r
\dd r=\int_{\tilde a_n}^{\tilde b_n}
 \minpartial{\cE}{t}{\tilde\teta_n(s)}\, \mdt{\tilde\teta} n s
\dd s.
\end{equation*}
We now define for every  $i=0,...,M_n-1$ the sets $
A_n^{i,\delta}:=\left\{s\in(\sigma_i^n,\sigma_{i+1}^n)\,:\,\tilde\teta_n(s)\in\mathrm{int}(K_{\delta})\right\}$,
where $K_{\delta}$ is defined
as in \eqref{set-Kdelta}, and write $A_n^{i,\delta}$ as the countable
union of its connected components, i.e. $
A_n^{i,\delta}=\bigcup_{k=1}^{\infty}(a^{i,\delta}_{n,k},b^{i,\delta}_{n,k}).
$ Similarly, we consider the analogues of the sets $B_n^\delta$
\eqref{def-b-enne-delta}, viz.
\begin{equation}
\begin{gathered}
B_n^{i,\delta}:=\bigcup_{(a^{i,\delta}_{n,k},b^{i,\delta}_{n,k}) \in
\mathfrak{B}_n^{i,\delta}} (a^{i,\delta}_{n,k},b^{i,\delta}_{n,k})
\quad \text{with}
\\
\begin{aligned}
 \mathfrak{B}_n^{i,\delta}=
\bigg\{(a^{i,\delta}_{n,k},b^{i,\delta}_{n,k})\subset
A_n^{i,\delta}\,:\, & \tilde\teta(a^{i,\delta}_{n,k})\in\pa B(x,\delta),\ \tilde\teta(b^{i,\delta}_{n,k})\in\pa B(y,\delta)
 \text{ for $x,\, y \in (\calC(t) {\cap} \subl\rho) $ with $x \neq y$} \bigg\},
\end{aligned}
\end{gathered}
\end{equation}
for  $i=1,..,M_n-1$. We denote by $N(i,n,\delta)$ the cardinality of
the set $\mathfrak{B}_n^{i,\delta}$. Then, we have
\begin{eqnarray*}
C \geq\int_0^1
 \minpartial{\cE}{t}{\teta_n(r)}\, \mdt{\teta} n r
\dd  r
&\geq&
e_{\delta}M_n\sum_{(a^{i,\delta}_{n,k},b^{i,\delta}_{n,k})\in\mathfrak{B}_n^{i,\delta} }\overline m,
\end{eqnarray*}
where
$0<\overline
m:=
\min_{n \in \N}\min_{
\stackrel{x,y\in(\calC(t){\cap} \subl\rho)\cup\{u_1,u_2\}}
{x\neq y}}
(\| x-y\|
-2\overline\delta)$, with $\overline\delta$ as in \eqref{delf-hat-delta} and  $C$ is as in \eqref{int_maggiorato}.
Therefore, we conclude  the estimate
\begin{equation*}
M_nN(i,n,\delta)\leq\frac C{e_{\delta}\overline m}\ \mbox{ for
every}\ 0<\delta\leq\overline\delta,\ \ \ n\in\N.
\end{equation*}
Observing that we may suppose $M_n, \ N(i,n,\delta)\geq 1$, we
conclude a bound for both  $(M_n)_n$  and
$((N(i,n,\delta))_{i=1}^{M_n})_n$.

The proof can be then carried out by suitably adapting the argument
for Proposition \ref{super-lemmone}. \end{proof}
We are now in the position to develop the
\paragraph{\bf Proof of Theorem \ref{prop:cost}.}
\textbf{Ad (1):} Suppose $\cost{t}{u_1}{u_2}=0$.
Then, by definition of $\cost{t}{u_1}{u_2}$,
there exists a sequence
$(\teta_n)_n\subset\admis{u_1}{u_2}{t}$ such that
\begin{equation*}
0=
\lim_{n \to \infty}\int_0^1
\minpartial{\cE}{t}{\teta_n(s)}\,\mdt{\teta} n s \dd s.
\end{equation*}
Then, it follows from
Prop.\ \ref{lemma:extension-lemmone} \emph{(1)}
that $u_1 = u_2$.

\par
\textbf{Ad (2):} consider the nontrivial case
$u_1\neq u_2$
and a curve $\teta\in\admis{u_1}{u_2}{t}$. Define $\zeta : [0,1]\to \Hilbert$ by
$\zeta(s) : = \teta(1-s)$. Then $\zeta \in\admis{u_2}{u_1}{t}$
 and
 \begin{align*}
\cost{t}{u_1}{u_2} \leq  \int_{0}^{1}
\minpartial \calE {t}{\teta(r)} \|\teta'(r)\| \dd r
= \int_{0}^{1}
\minpartial \calE {t}{\zeta(r)} \|\zeta'(r)\| \dd r\,.
\end{align*}
With this argument we easily conclude that  $\cost{t}{u_1}{u_2}\leq\cost{t}{u_2}{u_1}$.
Interchanging the role of $u_1$ and $u_2$ we conclude the
symmetry of the cost. 
\par
\textbf{Ad (3):}
We use the direct method of  the calculus of
variations: let $(\teta_n)_n\subset\admis{u_1}{u_2}{t}$
be a minimizing sequence for $\cost{t}{u_1}{u_2}(<\infty)$.
Applying Proposition \ref{lemma:extension-lemmone} \emph{(2)}
to the curves $\teta_n$
(in fact, we are in the case $u_1\neq u_2$), we conclude.

\par
\textbf{Ad (4):} 
We confine the discussion to the case in which  $\cost t{u_1}{u_3} >0$ and
$\cost t{u_3}{u_2} >0$,  
as the other cases can be treated with simpler arguments. 
Let $\teta_{1,3}$ and $\teta_{3,2}$ be two optimal
curves for $\cost t{u_1}{u_3} $ and
$\cost t{u_3}{u_2} $, respectively.
Set
 \[
 \teta_{1,2}(s):= \begin{cases}
 \teta_{1,3}(2s) & \text{for } s \in \left[ 0,\frac12 \right],\\
 \teta_{3,2}(2s-1) & \text{for } s \in \left(\frac12,1 \right].
 \end{cases}
 \]
Since $u_3\in\mathcal C(t)$, 
it is immediate to check that
$ \teta_{1,2}
\in \admis{u_1}{u_2}{t}$, 
and by the definition  of $\cost t{u_1}{u_2}$ we obtain
\[
\begin{aligned}
\cost t{u_1}{u_2}  \leq \int_0^1 \minpartial \cE t
{\teta_{1,2}(s)} \,\| \teta' (s)\| \dd s &  =   2 \int_0^{1/2}
\minpartial \cE t {\teta_{1,3}(2s)} \,\mdt {\teta}{1,3} {2s} \dd s
\\ &  +  2 \int_{1/2}^1  \minpartial \cE t {\teta_{3,2}(2s-1)}
\,\mdt {\teta}{3,2} {2s-1} \dd s\\ &  = \cost t{u_1}{u_3} +
\cost t{u_3}{u_2},
\end{aligned}
\]
and conclude \eqref{prop:triangle}.
\par
\textbf{Ad (5):}   \eqref{charact-cost_AC} is  a direct consequence of Proposition \ref{super-lemmone}. 
\par
\textbf{Ad (6):} We may suppose 
that $u_1 \neq u_2$ 
(otherwise, $\cost t{u_1}{u_2}=0$ and the desired inequality trivially follows),
and that $\liminf_{k\to\infty} \cost{t}{u_1^k}{u_2^k}<\infty$.
By definition of  $\cost{t}{u_1^k}{u_2^k}$, 
we have that for every $k\geq1$ there exists
$\teta_k\in\admis{u_1^k}{u_2^k}{t}$  such that
\begin{equation*}
 \int_0^1 \minpartial \calE {t}{\teta_k(s)}\|\teta_k'(s)\|\dd s
\,\leq\,
\cost{t}{u_1^k}{u_2^k}+\frac1k.
\end{equation*}
By Prop.\ \ref{lemma:extension-lemmone} (2),
there exists $\teta\in\admis{u_1}{u_2}{t}$ such that
\begin{equation*}
\int_0^1 \minpartial \calE {t}{\teta(s)}\|\teta'(s)\|\dd s
\,\leq\,
\liminf_{k\to\infty}
 \int_0^1 \minpartial \calE {t}{\teta_k(s)}\|\teta_k'(s)\|\dd s 
\,\leq\,
\liminf_{k\to\infty}
\cost{t}{u_1^k}{u_2^k}. 
\end{equation*}
This concludes the proof, in view of the definition of
$\cost{t}{u_1}{u_2}$.
\QED


\section{Proof of the main results}
\label{s:4}
\subsection{Proof of Theorem \ref{mainth:1}}
\label{ss:4.1}
Let $(u_{\eps_n})_n \subset H^1(0,T;\xfin)$ be a sequence of solutions to the Cauchy problem for \eqref{e:sing-perturb}, supplemented with initial data $(u_{\eps_n}^0)_n$ fulfilling \eqref{cvg-initial-data}.

\par\noindent
In the upcoming result we derive from the energy
identity \eqref{eqn_lemma} 
for the sequence
family $(u_{\eps_n})_n$, namely
\begin{equation}
\label{eqn_lemma_u_n}
\int_s^t\left(\frac{\eps_n}2\|u_{\eps_n}'(r)\|^2
+\frac1{2\eps_n}\|\rmD \ene r{u_{\eps_n}(r)} \|^2 \right)\dd r
+\mathcal E_t(u_{\eps_n}(t))=
\mathcal E_s(u_{\eps_n}(s))
+\int_s^t\pt r{u_{\eps_n}(r)}\dd r,
\end{equation}
a series of a priori estimates, 
which will allow us to prove a preliminary
compactness result, Proposition \ref{prop:1compactness} below.

\begin{proposition}[A priori estimates]
\label{u_ep_limit} 
Assume \eqref{Ezero}--\eqref{P_t}.
Then, there exists a constant $C>0$ such that for every
$n\in \N$ the following estimates hold
\begin{align}
& \label{bound-energie}
\sup_{t \in [0,T]} \cg{u_{\eps_n}(t)} + \sup_{t
\in [0,T]} |\pt  t{u_{\eps_n}(t)}| \leq C,
\\
& \label{bound-en-diss}
\int_s^t \left(\frac{{\eps_n}}2 \mdtq u{\eps_n} r
2+\frac1{2\eps_n} \minpartialq {\cE}{r}{u_{\eps_n}(r)}\right) \dd r
  \leq C \quad \text{for all } 0 \leq s\leq t \leq T.
\end{align}
\end{proposition}
\begin{proof}
Combining \eqref{eqn_lemma} with  estimate
\eqref{P_t} for the power function $\Ptname$,
 we find that
\begin{equation}\label{eqn_lemma2}
\begin{aligned}
\int_0^t \left(\frac{\eps_n}2 \mdtq u{\eps_n} s 2+\frac1{2{\eps_n}} \minpartialq
{\cE}{s}{u_{\eps_n}(s)}\right) \dd s +\ene {t}{u_{\eps_n}(t)} \leq \ene  0
{u_{\eps_n}^0}
       +C_1\int_0^t \ene  s {u_{{\eps_n}}(s)} \dd s +C_2T.
\end{aligned}
\end{equation} Now, in view of \eqref{cvg-initial-data} we have $\sup_n \ene  0
{u_{\eps_n}^0}
\leq C$. 
Hence, with  the Gronwall Lemma we conclude that $\sup_{t
\in [0,T]} \cg {u_{\eps_n}(t)}\leq C$, which in turn implies
\eqref{bound-energie}, in view of \eqref{P_t}. Therefore, we also
conclude \eqref{bound-en-diss}.
\end{proof}
\noindent The ensuing compactness result provides what will reveal
to be the \emph{defect measure} $\mu$ (cf.\ Definition
\eqref{def:sols-notions-1}) associated with  the limiting curve $u$ that shall be constructed later on. 
In what follows we will also show that the limiting energy and power functions $\limen$ and $\limp$, cf.\
\eqref{2converg} and \eqref{bis-2converg} below,
  coincide with the energy and power evaluated along $u$.
\begin{prop}
\label{prop:1compactness}
Assume
\eqref{Ezero}--\eqref{P_t}.
Consider the sequence of measures
\begin{equation}
\label{not-mu_n}
\mu_n:= \left(
\frac{\ep_n}2 \mdtq u{\eps_n}{\cdot}2
+\frac1{2\eps_n} \minpartialq \cE {(\cdot)}{u_{\eps_n}(\cdot )}
\right)\, \mathscr{L}^1,
\end{equation}
with $\mathscr{L}^1 $ the Lebesgue measure on $(0,T)$.  Then, there exist
a positive Radon measure $\mu \in \mathrm{M}(0,T)$ and functions $\limen\in
\BV ([0,T])$ and $\limp\in L^\infty (0,T)$ such that, along a not
relabeled subsequence, there hold 
as $n\to\infty$
\begin{align}
& \label{1converg}
\mu_n \weaksto \mu \quad \text{in }\mathrm{M}(0,T),
\\
& \label{2converg}
\lim_{n \to \infty} \ene t {u_{\eps_n}(t)}\,=\,
\limen(t) \quad \text{for all  } t \in [0,T],
\\
& \label{bis-2converg}
\pt t{u_{\eps_n}(t)} \weaksto \limp
\quad\text{in $L^\infty (0,T)$.}
\end{align}
Moreover,
denoting by $\limen_- (t)$ and $\limen_+ (t)$ the left
and right limits of $\limen$ at $ t \in [0,T]$, with the convention that 
$\limen_-(0): = \limen(0)$ and $\limen_+(T) : = \limen(T)$, 
we have that
\begin{align}
&\mu([s,t])+\limen_+(t)=\limen_-(s)
+\int_s^t \limp(r)\dd r
\ \qquad\mbox{for every }\ 
0\leq s \leq t\leq T.
\label{identita-energia-limen}
\end{align}
\par
Furthermore, denoting by
$\mathrm{d} \limen$ the distributional derivative of $\limen$, 
we get from the previous identitities that 
\begin{align}
&\label{left_right_limits}
\limen_- (t)-\limen_+ (t)=\mu(\{t\})
\qquad\mbox{for every }
0 \leq t \leq T; \\
&\label{identity-between-measures} \mathrm{d} \limen + \mu= \limp
\mathscr{L}^1\,.
\end{align}
\par
Finally, let $J$  be the set where the measure $\mu$ is atomic. Then
\begin{equation}
\label{enhanced-prop-3_var} J= \{t\in[0,T]\,:\,\mu (\{t\})>0 \}
\text{ consists of at most countably infinitely many points}.
\end{equation}
\end{prop}
\begin{proof}
It follows from estimate \eqref{bound-en-diss} in Proposition
\ref{u_ep_limit} that the measures $(\mu_n)_n$ have uniformly
bounded variation, therefore \eqref{1converg} follows. As for
\eqref{2converg}, we observe that, by \eqref{eqn_lemma2}, the maps
$t\mapsto \mathscr{F}_n(t):=  \ene t{u_{\eps_n}(t)}-\int_0^t \pt  s
{u_{\eps_n}(s)} \dd s $ are nonincreasing on $[0,T]$. Therefore, by
Helly's Compactness Theorem
there exist $\mathscr{F} \in \BV([0,T])$ such that, up to a
subsequence, $\mathscr{F}_n(t) \to \mathscr{F}(t)$ for all $t \in
[0,T]$. On the other hand, \eqref{bound-energie} also yields
\eqref{bis-2converg}, up to a subsequence. Therefore,
\eqref{2converg} follows with $ \limen(t):= \mathscr{F}(t)+\int_0^t
\limp(s)\dd s.$

To prove identity \eqref{identita-energia-limen}, 
let us first suppose, for simplicity, that $0<s\leq t<T$. We note on the one hand that 
$[s,t]=\bigcap_m(s-1/m,t+1/m)$. Hence,
from the fact that $\mu([s,t])=\lim_{m\to\infty} \mu\big((s-1/m,t+1/m)\big)$
and from convergence \eqref{1converg}, we get
\begin{align}
\mu([s,t])
&=
\lim_{m\to\infty} \mu\big((s-1/m,t+1/m)\big)
\leq
\lim_{m\to\infty}\liminf_{n \to\infty} \mu_n\big((s-1/m,t+1/m)\big)\nonumber\\
&=
\lim_{m \to\infty} \liminf_{n \to\infty} \mu_n\big([s-1/m,t+1/m]\big)\nonumber\\
&=
\lim_{m \to\infty}
\left\{
\limen (s-1/m)-\limen(t+1/m)+\int_{s-1/m}^{t+1/m}
\limp(r)\dd r\right\}
=
\limen_-(s)-\limen_+(t).\label{disug1}
\end{align}
Note that in the last equality we have used the fact 
that the limits $\limen_-(s)$ and $\limen_+(t)$ 
always exist,
since $\limen \in \BV ([0,T])$.
On the other hand, since the identity
$[s,t]=\bigcap_m[s-1/m,+1/m]$
holds as well,
we have at the same time that
\begin{align}
\mu([s,t])
&=
\lim_{m \to\infty} \mu\big([s-1/m,t+1/m]\big)
\geq
\lim_{m \to\infty} \limsup_{n \to\infty} \mu_n\big([s-1/m,t+1/m]\big)\nonumber\\
&=
\lim_{m \to\infty}
\left\{
\limen (s-1/m)-\limen(t+1/m)+\int_{s-1/m}^{t+1/m}
\limp(r)\dd r\right\}
=
\limen_-(s)-\limen_+(t). \label{disug2}
\end{align}
From inequalities \eqref{disug1} and \eqref{disug2}
we obtain \eqref{identita-energia-limen}.
With  obvious modifications we can handle the cases $s=0$ and $t=T$. 
Identity \eqref{identity-between-measures} 
trivially follows from \eqref{identita-energia-limen}.

Finally, let us denote by $(\mathrm{d} \limen)_{\mathrm{jump}}$ the
jump part of the measure $\mathrm{d} \limen$: it follows from
\eqref{identity-between-measures} that
\begin{equation}
\label{interesting-observation}
\mathrm{supp}((\mathrm{d}
\limen)_{\mathrm{jump}}) = J.
\end{equation}
Then, \eqref{enhanced-prop-3_var} follows from recalling that
$\limen \in \BV([0,T])  $ has countably many jump  points.
\end{proof}

\begin{notation}
\label{not-J}\upshape
Hereafter, we will denote by $B$ the set
\begin{equation}
\label{set-B} B= \{t \in (0,T)\,:\,\minpartial \cE
{t}{u_{{\ep}_n}(t) } \to 0 \text{ as } n \to \infty\}
\end{equation}
where $(u_{\eps_n})_n$ is (a suitable subsequence of) 
the sequence for which convergences
\eqref{1converg}--\eqref{bis-2converg} hold.
Due to \eqref{bound-en-diss}, we have that
\[
\lim_{n \to\infty}\int_0^T
\minpartialq\cE{r} {u_{{\ep}_n}(r)} \dd r=0,
\]
hence the set $B$   (defined for a suitable subsequence),
has full Lebesgue measure.
\end{notation}
\par
While the compactness statements in Proposition \ref{prop:1compactness} only relied on assumptions \eqref{Ezero}--\eqref{P_t},
for the next result, which will play a key role in the compactness argument within the proof of Theorem \ref{mainth:1}, 
we additionally need condition \eqref{strong-critical} 
on the critical points of $\calE$. 
\begin{lemma}
\label{lemma:convergence-to-critical}
Assume   \eqref{Ezero}--\eqref{strong-critical}.
For every $t \in [0,T]$ and for all sequences
$(t_1^n)_n$,  $(t_2^n)_n$ fulfilling
$0\leq t_1^n \leq t_2^n \leq T$ for every $n\in\N$
and 
\begin{equation}
\label{hyp-lemma-conv}
t_1^n\to t, \quad t_2^n \to t, \quad
u_{\eps_n}(t_1^n) \to u_1, \quad  u_{\eps_n}(t_2^n) \to u_2 \text{ for some } u_1,\, u_2 \in \xfin,
\end{equation}
there holds
\begin{equation}\label{mu_costo}
\mu(\{t\})\geq \cost t{u_1}{u_2}\,.
\end{equation}
In particular, for every
$t \in [0,T] \setminus J$  we have that $u_1=u_2$.
\end{lemma}
\begin{proof}
Observe that for every $\eta>0$ there holds
\begin{equation}
\label{to-be-cited-later}
\begin{aligned}
\mu ([t-\eta, t+\eta])&\geq\limsup_{n \to \infty}\mu_n ([t_1^n, t_2^n])
\\ &
=
 \limsup_{n \to \infty}  \int_{t_1^n}^{t_2^n} 
 \left( \frac{\ep_n}2\|u_{\ep_n(s )}'\|^2
+\frac1{2\eps_n} \| \rmD \ene {s}{u_{{\ep}_n} (s ) }\|^2 \right)  \dd s
\\ &
\geq
\limsup_{n \to \infty}  \int_{t_1^n}^{t_2^n}  
\|u_{\ep_n(s )}'\| \| \rmD \ene {s}{u_{{\ep}_n} (s ) }\|  \dd s 
\geq \cost{t}{u_1}{u_2},
\end{aligned}
\end{equation}
where the first inequality is due to \eqref{1converg}, the second one to the definition \eqref{not-mu_n} of $\mu_n,$
the third one to the Young inequality, and the last one to
 \eqref{charact-cost_AC} in Proposition \ref{prop:cost}.
Since $\eta>0$ is arbitrary, we conclude \eqref{mu_costo}.
In particular, if  $\mu(\{t\})=0$
then by (1) in Proposition \ref{prop:cost} we  deduce that $u_1=u_2$.
\end{proof}
\par
We are now in the position to perform the \textbf{\underline{proof of Theorem \ref{mainth:1}}:} we will split the arguments in several points.
\par
\noindent
\textbf{Ad \eqref{3converg_var}:}
Let us  consider the set (cf.\ Notation \ref{not-J})
\begin{equation}\label{def-I_var}
I:= J \cup A \cup \{0\}\ \text{with }A\subset(B{\setminus}J)\text{ dense in } [0,T]
\text{ and consisting of countably many points}.
\end{equation}
From 
\eqref{bound-energie}   we gather 
that
\begin{equation}
\label{4-ease-later}
\exists\, C>0 \ \ 
 \forall\, n \in \N, \ \forall\, t \in [0,T]\, : \quad
u_{\ep_n}(t )\in \subl C, \text{ with } \subl C \Subset \xfin \text{ by \eqref{coercivita}.}
\end{equation}
Since $I$ has countably many points,
with a diagonal procedure it is possible to extract from $( u_{{\ep}_n})_n$ a (not relabeled) subsequence
such that there exists $\hat u : I \to X$ with
\begin{equation}\label{eq:convergence-I_var}
u_{\eps_n}(t) \to\hat u(t) \quad\text{for all } t \in I,
\end{equation}
with $\hat{u}(0)=u_0$ thanks to 
the convergence \eqref{cvg-initial-data} of the initial data. 
Moreover, since $A \subset B$ from \eqref{set-B}, we also have 
\begin{equation}\label{critici-1_var}
\hat u(t) \in \mathcal C(t)\qquad\mbox{for every }t \in  A.
\end{equation}
\par
We now extend $\hat u$ to a function defined on the whole interval $[0,T]$,
by showing that 
\begin{equation}\label{limite_unico}
\begin{gathered} \forall\, t\in(0,T]\setminus I \quad
\tilde u(t):=\lim_{k\to\infty}\hat u(t_k)   \text{ is uniquely defined for every $(t_k)_k\in \mathfrak{S}(t)$ and fulfills } \tilde{u}(t) \in \calC(t),
\\
\text{with}\qquad
\mathfrak{S}(t)= 
\left\{(t_k)_k \subset A\, : \  t_k\to t \text{ and } \exists\, \lim_{k\to\infty}\hat u(t_k)\right\}
\end{gathered}
\end{equation}
(in the case $t=T$, the sequence $(t_k)_k$ is to be understood as $t_k\uparrow t$). 
 Observe that $\mathfrak{S}(t) \neq \emptyset$  
since $\hat{u}(I)$ is contained in the compact set $K$ from 
\eqref{4-ease-later},
To  check \eqref{limite_unico}, let $(t_1^k)_k, \ (t_2^k)_k\in \mathfrak{S}(t) $ be such that
\begin{equation*}
\lim_{k\to\infty}\hat u(t_1^k)=:u_1\qquad\mbox{and}\qquad\lim_{k\to\infty}\hat u(t_2^k)=:u_2.
\end{equation*}
We want to show that $u_1=u_2$. Note that $\hat u(t_1^k)=\lim_{n\to\infty}u_{\eps_n}(t_1^k)$
and $\hat u(t_2^k)=\lim_{n\to\infty}u_{\eps_n}(t_2^k)$ for every $k\in\N$, because of \eqref{eq:convergence-I_var}.
Since $t_1^k$, $t_2^k\in A \subset B $ for every $k \in \N$,  there holds
$\rmD\ene{t_1^k}{\hat u(t_1^k)}=\rmD\ene{t_2^k}{\hat u(t_2^k)}=0$ for every $k \in
\N$. Therefore,  we get that  $u_1$, $u_2\in\mathcal C(t)$.
Furthermore,
with a diagonal procedure we can extract a subsequence $(n_k)_k$ such that
\begin{equation*}
u_1=\lim_{k\to\infty}u_{n_k}(t_1^k)\qquad\mbox{and}\qquad u_2=\lim_{k\to\infty}u_{n_k}(t_2^k).
\end{equation*}
Therefore, we are in the position to apply  Lemma \ref{lemma:convergence-to-critical}
to $u_1$ and $u_2 $. Since
$t \notin J$, we have that
$u_1= u_2$. This concludes the proof of \eqref{limite_unico}.
Therefore, we can define  the (candidate) limit function $u$  everywhere on $[0,T]$ by setting 
\begin{equation}\label{defi_u}
u(t):=
\left\{
\begin{array}{ll}
\hat u(t) & \mbox{if }t\in I\\
\tilde u(t) & \mbox{if }t\in(0,T]\setminus I.
\end{array}
\right.
\end{equation}
By construction, $u$ complies with \eqref{to-save-1}. 

We now address the pointwise convergence  \eqref{3converg_var}: in view of \eqref{eq:convergence-I_var},  we have to show it at 
 $t\in (0,T]\setminus I$. We  will prove that at any such point $t$, any subsequence of $(u_{\eps_n}(t))_n$ admits
a further subsequence converging to $u(t)$.
Let us fix a (not relabeled) subsequence $(u_{\eps_n}(t))_n$ and consider a sequence $(t_k)_k\subset A$
such that $t_k\uparrow t$ and $u(t)=\tilde u(t)=\lim_{k\to\infty}\hat u(t_k)$. 
With a diagonal procedure as in the above lines,
we find a subsequence $({\eps}_{n_k})_k$ such that
\begin{equation*}
u(t)=\lim_{k\to\infty}u_{{\eps}_{n_k}}(t_k),
\end{equation*}
whereas, again using that $u_{\eps_{n_k}}([0,T]) \subset \subl C$ for every $k\in \N$, we extract a further (not relabeled) subsequence 
from $(u_{{\eps}_{n_k}}(t))_k$, 
such that 
$$
u_{{\eps}_{n_k}}(t)\to \tilde u \quad \text{for some } \tilde u \in \xfin.
$$
 Since $t\notin J$, 
an application of Lemma \ref{lemma:convergence-to-critical} with $t_{k}$, $t$, $u_{\eps_{n_k}}$,
$u(t)$, and $\tilde u$ in place of $t_1^n$, $t_2^n$, $u_{\ep_n}$, $u_1$, and $u_2$,
respectively, gives that $\tilde u =u(t)$.
Therefore,  
convergence \eqref{3converg_var} holds at $t\in (0,T]\setminus J$,
and at $t\in J$ due to \eqref{eq:convergence-I_var}
and definition \eqref{defi_u}.
\par
\noindent \textbf{Ad \eqref{4converg_var}:}
Since the sequence $(u_{\eps_n})_n $ is bounded in $L^\infty (0,T; \xfin)$ by 
\eqref{4-ease-later}, 
 \eqref{4converg_var} follows from  \eqref{3converg_var}. 
 \par
 It now remains to verify that $u \in \mathrm{B}([0,T];\xfin)$ complies with the properties 
\eqref{to-save-eneq}----\eqref{to-save-complex} 
defining the notion of Dissipative Viscosity solution. 
\par
\noindent \textbf{Ad  \eqref{to-save-eneq}:}
To prove \eqref{to-save-eneq}, we first need to prove that
the left and the right limits of $u$ always exist.
We now show that for every $0\leq t<T$ the right
limit $u_+(t)$ exists. The same argument
can be trivially adapted to prove the existence of the left limit   
$u_-(s)$ for every $s\in (0,T].$ 
Consider $(t_1^k)_k$, $(t_2^k)_k\subset [0,T]$
such that $t_1^k\downarrow t$, $t_2^k\downarrow t$, and the
limits
\begin{equation}
\label{added-label}
\lim_{k\to\infty}u(t_1^k)=:u_1\qquad\qquad\lim_{k\to\infty}u(t_2^k)=:u_2
\end{equation}
exist.
Note that, up to subsequences, we have that either
$t_1^k\leq t_2^k$ or $t_2^k\leq t_1^k$  for every $k\in\N$.
Suppose for simplicity that we are in the first case.
Observe that from \eqref{Ezero} and from \eqref{2converg},
we have that $\ene t{u(t)}= \limen(t)$
for every $t\in[0,T]$, due to 
convergence \eqref{3converg_var}.
In particular, since $\limen\in\BV([0,T])$, 
there exist
\begin{equation}
\label{lim_da_dx}
\lim_{k\to\infty} \ene {t_1^k}{u(t_1^k)}
\,=\, 
\lim_{k\to\infty} \ene {t_2^k}{u(t_2^k)} = \limen_+(t).
\end{equation}

Now, \eqref{2converg}
gives that
  $\ene {t_i^k}{u(t_i^k)}=\lim_{n\to\infty}\ene{t_i^k}{u_{\eps_n}(t_i^k)}$ for every
$k \in \N$ and for $i=1,2$. Hence, there exists $({\eps_{n_k}})_k$ such that
\begin{equation*}
\left|\ene{t_1^k}{u_{\eps_{n_k}}(t_1^k)}-  \ene {t_1^k}{u(t_1^k)} \right|\leq\frac1 k,\qquad
\left| \ene{t_2^k}{u_{\eps_{n_k}}(t_2^k)}-  \ene {t_2^k}{u(t_2^k)}  \right|\leq\frac1 k
\end{equation*}
for every $k\in\N$, so that
\begin{equation}\label{limiti_uguali}
\lim_{k\to\infty} \ene{t_1^k}{u_{\eps_{n_k}}(t_1^k)}=\lim_{k\to\infty} \ene {t_1^k}{u(t_1^k)} ,\qquad
\lim_{k\to\infty} \ene{t_2^k}{u_{\eps_{n_k}}(t_2^k)}=\lim_{k\to\infty} \ene {t_2^k}{u(t_2^k)} \,.
\end{equation}
Arguing as previously done, we can also suppose that, up to a subsequence,
\begin{equation}
\label{u1u2-limiti}
u_1=\lim_{k\to\infty}u_{\eps_{n_k}}(t_1^k),
\qquad\qquad 
u_2=\lim_{k\to\infty}u_{\eps_{n_k}}(t_2^k).
\end{equation}
Now, recalling definition \eqref{not-mu_n} of $\mu_n$, the energy identity \eqref{eqn_lemma} with
$t_1^k$, $t_2^k$, $u_{\eps_{n_k}}$ in place
of $s$, $t$, and $u_{\ep}$, respectively, gives
\begin{equation*}
\mathcal E_{t_1^k}(u_{\eps_{n_k}}(t_1^k))-\mathcal E_{t_2^k}(u_{\eps_{n_k}}(t_2^k))
+\int_{t_1^k}^{t_2^k} \pt r{u_{\eps_{n_k}}(r)}\dd r=\mu_{{n_k}}([t_1^k,t_2^k]).
\end{equation*}
This equality, together with \eqref{limiti_uguali},  \eqref{u1u2-limiti},
and with \eqref{charact-cost_AC} in 
Theorem \ref{prop:cost}, 
implies that
\begin{equation}\label{disug-limen-costo_1}
0=
\lim_{k\to\infty} \ene {t_1^k}{u(t_1^k)} -\lim_{k\to\infty} \ene {t_2^k}{u(t_2^k)} 
\geq\liminf_{k\to\infty}\mu_{\eps_{n_k}}([t_1^k,t_2^k])\geq \cost t{u_1}{u_2}
\end{equation}
(note that we have also used \eqref{bis-2converg}
and \eqref{lim_da_dx}).
Hence, we have obtained that $\cost t{u_1}{u_2}=0$
and in turn that $u_1=u_2$, in view of Proposition \ref{prop:cost} (1), 
whence we conclude 
that the right limit
$u_+(t)$ exists.

Combining  
\eqref{3converg_var} with \eqref{2converg}, 
and taking into account that $\calE \in \rmC^1([0,T]\times \xfin)$,  
we gather that
 \[
 \mathscr{E}_+(t) = \ene t{u_+(t)}\quad\text{for all }
 0\leq t<T,
 \qquad
 \mathscr{E}_-(s) = \ene s{u_-(s)}\quad\text{for all }
 0< s\leq T, 
 \]
 and
\[
\pt t{u_{\eps_n}(t)} \to \pt t{u(t)} \qquad \text{for all } t \in [0,T].
 \]
In view of  \eqref{bound-energie} and the Lebesgue theorem, 
we then have $\pt t{u_{\eps_n}(t)} \to \pt t{u(t)} $ in $L^p(0,T)$ for every $1\leq p<\infty$. Therefore, 
 \[
 \mathscr{P}(t) = \pt t{u(t)} \quad \foraa\, t \in (0,T),
 \]
and  the energy balance \eqref{to-save-eneq} follows from \eqref{identita-energia-limen}.

\par
\noindent \textbf{Ad \eqref{to-save-2}:}
To prove that  $u_+(t) \in \calC(t)$ 
for every $t \in [0,T)$ 
(the argument for $u_-(t)$, with $t\in (0,T]$,  is perfectly analogous),
it is sufficient to observe that there always
exists $t^k\downarrow t$ such that $(t^k)_k\subset(0,T]\setminus J$,
so that in particular
\[
u_+(t)=\lim_{k\to\infty} u(t^k),
\]
and $u(t^k)\in\calC(t^k)$ for every $k\in\N$.
Therefore, by this limit and by \eqref{Ezero}, 
$u_+(t)\in\calC(t)$.  

\noindent 
\textbf{Ad \eqref{to-save-3}\&\eqref{to-save-4}:} 
preliminarily, we show that
\begin{equation}\label{disug-limen-costo_2}
\ene {t}{u_-(t)}- \ene {t}{u_+(t)} \geq\cost t{u_-(t)}{u_+(t)} \quad \text{for every } t \in (0,T)
\end{equation}
(suitable analogues hold at the points $t=0$ and $t=T$). Indeed, 
fix
$t_1^k\uparrow t$ and $t_2^k\downarrow t$,
so that  (cf.\ \eqref{added-label}) $u_1= \lim_{k\to\infty} u(t_1^k)=u_-(t)$ and  $u_2= \lim_{k\to\infty} u(t_2^k)=u_+(t)$. 
The very same arguments leading to  \eqref{disug-limen-costo_1} show that 
\[
\lim_{k\to\infty} \ene {t_1^k}{u(t_1^k)} -\lim_{k\to\infty} \ene {t_2^k}{u(t_2^k)} 
\geq\liminf_{k\to\infty}\mu_{\eps_{n_k}}([t_1^k,t_2^k])\geq \cost t{u_1}{u_2}\,.
\]
Then, \eqref{disug-limen-costo_2} ensues. 
On account of identity (\ref{left_right_limits}), we deduce 
\begin{equation}
\label{prelim-added}
\mu(\{t\})\geq \cost t{u_-(t)}{u_+(t)} \quad \text{for every } t \in [0,T]\,.
\end{equation}
In particular, if $t\notin J$, 
 we have
  $\cost t{u_-(t)}{u_+(t)}=0$, hence
$u_-(t)=u_+(t)$. Thus, we have proved the one-sided inclusion  $\supset$  in 
\eqref{to-save-3}.
%
\par
Let us now prove the converse of  inequality \eqref{disug-limen-costo_2},
namely
\[
\ene {t}{u_-(t)}- \ene {t}{u_+(t)} \leq\cost t{u_-(t)}{u_+(t)}.
\]
We may confine the discussion to the case $t\in J$ for, otherwise, 
we have  $u_-(t)=u_+(t)$ and the above inequality trivially holds.  
Let $\teta\in\admis{u_-(t)}{u_+(t)}{t}$
be a minimizing curve for the cost $\cost t{u_-(t)}{u_+(t)}$:
its existence is guaranteed by Theorem \ref{prop:cost} (3). 
Then, by the chain rule 
\begin{equation}
\label{converse}
\begin{aligned}
\cost t{u_-(t)}{u_+(t)}&=\int_{0}^1 \|\rmD\mathcal
E_t(\teta(s))\| \|\teta'(s)\| \dd s  \\
&\geq- \left(\ene t {\teta(1)} -\ene t {\teta(0)}\right)= \ene {t}{u_-(t)} - \ene t{u_+(t)}.
\end{aligned}
\end{equation}
All in all, again taking into account \eqref{left_right_limits},
we have proved that 
\[
\cost t{u_-(t)}{u_+(t)} = \mu(\{t\})=  \ene {t}{u_-(t)} - \ene t{u_+(t)} \quad \text{for all } t \in [0,T],
\]
whence \eqref{to-save-3}\&\eqref{to-save-4} also in view of Thm.\
 \ref{prop:cost}(1). 
\par
 This concludes the proof of Theorem \ref{mainth:1}.
\QED



\subsection{Proof of Theorem \ref{mainth:2}}
\label{ss:4.2}
Let us denote by $\mu_{\mathrm{AC}}$, $\mu_{\mathrm{J}}$, and $\mu_{\mathrm{CA}}$, the absolutely continuous, jump, and Cantor parts of the defect measure $\mu$. Recall that 
$
\mu_{\mathrm{J}}([s,t]) = \sum_{r\in J \cap [s,t]} \cost r{u_-(r)}{u_+(r)}
$
in view of \eqref{to-save-4}.   
It follows from \eqref{to-save-eneq} that, 
for every $0\leq t\leq T$, 
\[
\mu_{\mathrm{J}}([0,t])  +\ene t{u_+(t)}  
\leq 
\mu_{\mathrm{AC}}([0,t]) +\mu_{\mathrm{J}}([0,t]) +\mu_{\mathrm{CA}}([0,t]) 
+\ene t{u_+(t)}  
=
\ene 0{u(0)} + \int_0^t \pt r{u(r)} \dd r \,.
\]
We will now show that 
\begin{equation}
\label{reverse-eneq}
\mu_{\mathrm{J}}([0,t])  +\ene t{u_+(t)} 
\geq 
\ene 0{u(0)} + \int_0^t \pt r{u(r)} \dd r \quad 
\text{ for every }\ 0\leq t \leq T, 
\end{equation}
and therefore conclude that $\mu_{\mathrm{AC}}([0,t])= \mu_{\mathrm{CA}}([0,t])=0$ at every $t\in[0,T]$, whence the thesis.
\par
We will deduce \eqref{reverse-eneq} by applying the following result, which is a variant of \cite[Lemma 6.2]{Savare-Minotti}.
\begin{lemma}
\label{l:ftom_VE}
Let $g :[a,b]
\to \R$  be a strictly increasing function, and $f:[a,b]\to\R$ be right continuous and such that its restriction  to the set $[a,b]\setminus J_g $ is lower semicontinuous.
Suppose that 
\begin{align}
&
\label{liminf-cond}
\liminf_{r\uparrow t} f(r) - f(t) \geq g_-(t) - g_+(t) && \quad \text{for all } t \in J_g,
\\
& 
\label{limsup-cond}
\limsup_{s\downarrow t} \frac{f(t)-f(s)}{g_+(s) - g_+(t)} \geq -1 && \quad \text{for all } t \in [a,b]\,.
\end{align}
Then, the map $f - g$ is non-increasing on $[a,b]$.
\end{lemma}
\noindent In fact, \cite[Lemma 6.2]{Savare-Minotti} has the same thesis as the result above, but  `specular' conditions on $f$ and $g$,
involving left continuity of $f$,    the $\limsup$ from the right, 
in place of the $\liminf$ from the left, in \eqref{liminf-cond}, 
and analogously for \eqref{limsup-cond}, etc.
For our purposes, though, it is more convenient to apply the present version of the result. Its proof can be deduced from that of  \cite[Lemma 6.2]{Savare-Minotti} by observing that, for $f - g$ to be non-increasing on $[a,b]$, it is sufficient to have 
\[
f(b) -g(b) \leq f(a) - g(a)  \ \Leftrightarrow \  f^{\#}(a) - g^{\#}(a) \leq  f^{\#}(b) - g^{\#}(b) \text{ with } f^{\#}(t): = f(b+a-t), \, g^{\#}(t): = g(b+a-t)\,.
\]
Therefore, we  are led to prove  that the function $f^{\#}-g^{\#}$ is non-decreasing, which follows from applying  \cite[Lemma 6.2]{Savare-Minotti} to the functions $-f^{\#}$ and $-g^{\#}$. Rewriting the conditions on  $-f^{\#}$ and $-g^{\#}$ from   \cite[Lemma 6.2]{Savare-Minotti}  in terms of the original functions $f$ and $g$, one obtains the statement of Lemma \ref{l:ftom_VE}.
\par
We are now in the position to conclude the \textbf{\underline{proof of Theorem \ref{mainth:2}}}. Mimicking the argument from the proof of \cite[Thm.\ 6.5]{Savare-Minotti}, in order to conclude the lower energy estimate 
 \eqref{reverse-eneq} we shall apply Lemma \ref{l:ftom_VE} to the functions $f,\, g: [0,T]\to \R$ defined by 
\begin{equation}
\label{f&g}
f(t): = \int_0^t \pt r{u(r)} \dd r - \ene t{u_+(t)}  \quad\text{ and } \quad g(t) := \sum_{r\in J \cap [0,t]} \cost r{u_-(r)}{u_+(r)} + \eta t 
\end{equation}
with $\eta>0$ arbitrary, so that $g$ is strictly increasing. By construction $f$ is right continuous. Since $u$ is continuous on $[0,T]\setminus J_g$, and since $\calE \in \rmC^1([0,T]\times X)$, we deduce that $f$ is even continuous on $[0,T]\setminus J_g$.
To check \eqref{liminf-cond}, we observe that 
\[
\begin{aligned}
\liminf_{r\uparrow t} f(r) - f(t)  &  = \lim_{r\uparrow t } \int_t^r \pt{\tau}{u(\tau)} \dd \tau + \liminf_{r\uparrow t}  \ene t{u_+(t)}  - \ene r{u_+(r)}  
\\ & 
= \ene t{u_+(t)} - \ene t{u_-(t)} = - \cost{t}{u_-(t)}{u_+(t)} 
=  g_-(t) - g_+(t) \,.
\end{aligned}
\]
Note that in the last equality we have used the 
fact that
\[
g_+(t) - g_-(t)
=
\lim_{\tau\downarrow t}
\lim_{s\uparrow t}
\sum_{r\in J \cap [s,\tau]} \cost r{u_-(r)}{u_+(r)}
 = 
\cost t{u_-(t)}{u_+(t)}.
\]
Finally, in order to verify \eqref{limsup-cond}, we 
preliminarily
calculate 
\[
\begin{aligned}
f(t) -f(s)  & = \int_s^t \pt r{u(r)} \dd r+ \ene s{u_+(s)}  - \ene t{u_+(t) } \\ & = \int_s^t \left( \pt r{u(r)} {-} \pt r{u_+(s)} \right) \dd r+   \ene t{u_+(s)} - \ene t{u_+(t)} \doteq I_1+I_2\,.
\end{aligned}
\]
Observing  that 
\begin{equation}
\label{to-be-used}
g_+(s) - g_+(t) \geq \eta(s-t),
\end{equation}
we find that 
\[
\left| \frac{I_1}{g_+(s) - g_+(t)} \right| = \frac{|I_1|} {g_+(s) - g_+(t)} \leq \frac1\eta \sup_{r\in [s,t]} |  \pt r{u(r)} {-} \pt r{u_+(s)} | 
\to 0 \quad  \text{ as } s\downarrow t\,,
\]
due to the continuity of the map $(t,u) \mapsto \pt tu$. 
Therefore,
\[
\begin{aligned}
\limsup_{s\downarrow t} \frac{f(t)-f(s)}{g_+(s) - g_+(t)} = \limsup_{s\downarrow t} \frac{ \ene t{u_+(s)} - \ene t{u_+(t)} }{g_+(s) - g_+(t)}  
=  \limsup_{s\downarrow t} \frac{\ene t{u_+(s)} - \ene t{u_+(t)}} {\| \rmD \ene t{u_+(s)} \|} \frac{\| \rmD \ene t{u_+(s)} \|}{g_+(s) - g_+(t)} \geq 0\,,
\end{aligned}
\]
which follows from condition \eqref{Loja}, and from the fact that $\frac{\| \rmD \ene t{u_+(s)} \|}{g_+(s) - g_+(t)} \geq 0$ for all $s\geq t$ since $g_+$ is strictly increasing.
\par
All in all, 
\eqref{reverse-eneq}  ensues from writing $f(t) - g(t) \leq f(0)- g(0)$ with $f$ and $g$ from \eqref{f&g}, and letting $\eta \downarrow 0$. 
\QED


\section{Examples and applications}
\label{s:5}

In this section, we discuss two classes of conditions which
guarantee the validity for $\ene t{\cdot}$, $t\in [0,T]$,
of hypothesis \eqref{strong-critical}
on the set of its critical points,
and of the {\L}ojasiewicz inequality
\eqref{true-Loja}, respectively.
\par
We start by introducing the 
\emph{transversality conditions},
concerning the properties of the energy $\cE$ at 
points $(t,u)$ where $u$ is a \emph{degenerate} critical point, 
i.e. on the set
\begin{equation}
\label{singular-points}
\mathcal{S} := \big\{ (t,u) \in [0,T]\times\Hilbert:\ u\in\calC(t)\
\mbox{ and }\ \gder 2{\cE} tu 
\ \text{is non-invertible}\big\}
\end{equation}

\begin{definition}
\label{def:transv}
We say that the functional $\cE$ satisfies the 
transversality conditions if each point
$(t_0,u_0) \in \mathcal{S}$ fulfills 
\begin{enumerate}[\rm (T1)]
\item
  $\mathrm{dim}(\NN (\gder 2\cE {t_0}{u_0}))=1$;
\item
  If $0\neq v\in \NN(\gder 2\cE {t_0}{u_0})$
  then $\la\partial_t
  \gder{}{\cE}{t_0}{u_0}, v\ra  
  \neq 0$;
\item
  If $0\neq v\in \NN(\gder 2\cE {t_0}{u_0})$ then
  $
  \gder 3\cE {t_0}{u_0}[v,v,v]
    \neq 0$,
\end{enumerate}
where $\NN(\gder 2\cE {t_0}{u_0})$ denotes the kernel of the mapping 
$\gder 2\cE {t_0}{u_0}.$
\end{definition}

Under the transversality conditions we have the following result,
proved in \cite[Corollary 3.6]{genericity}, ensuring that 
the critical set $\calC(t)$ is discrete at every $t\in [0,T]$.

\begin{proposition}[\cite{genericity}]
\label{prop:7.3}
Let $\cE\in \rmC^3([0,T]\times\xfin)$  
comply with the transversality conditions.
Then, for every $t\in [0,T]$, 
the set $\calC(t)$ consists
of isolated points.
Hence, \eqref{strong-critical} holds.
\end{proposition}

With the following \emph{genericity} result, proved in 
\cite[Theorem 1.3. and Corollary 3.7]{genericity}, 
we see that suitable second-order perturbations of an arbitrary
energy functional lead to an energy fulfilling the transversality conditions.
In order to state it, we need to introduce 
the set ${\rm Sym}(\xfin)$ of the symmetric
bilinear forms on $\xfin\times\xfin$.
Moreover, for a further technical reason that we do not detail here, 
in the following theorem we have to require $\calE \in \rmC^4 ([0,T]\times\xfin)$.

\begin{theorem}[\cite{genericity}]
\label{thm:strong-perturb}
Let $\calE$ be in $\rmC^4 ([0,T]\times\xfin)$.
Then, every open neighborhood $U$ of the origin in 
$\xfin\times{\rm Sym}(\xfin)$ contains a set
$U_r$ of full Lebesgue measure such that, 
for every $(y, \mathscr{K})\in U_r$, the functionals
\begin{equation}
\label{lo-perturbed-2}
(t,u)
\,\longmapsto\,
\ene tu +  \pairing{}{}{y}{u}   + \frac12 \mathscr{K}(u,u)
\end{equation}
satisfy the transversality conditions.
\end{theorem}

Let us mention that in \cite{genericity}
a similar result (cf.\ \cite[Cor.\ 3.7]{genericity})   
is proved in a more general,
infinite-dimensional setting, with  perturbations
of the same form as \eqref{lo-perturbed-2},
fulfilling an infinite-dimensional version
of the transversality conditions. 
Such perturbations are constructed by means of elements 
$(y,\mathscr{K})\in(\xfin{\times}{\rm Sym}(\xfin))\setminus N$, where
$N$ is in general only a \emph{meagre} subset of 
$\xfin\times{\rm Sym}(\xfin)$. 
In the present finite-dimensional context,
$N$ meagre improves to an $N$ with zero Lebesgue measure,
due to the classical Sard's Theorem.


Concerning the {\L}ojasiewicz inequality, we are now
going to point out its connections with the 
concept of \emph{subanaliticity}. 
For the reader's convenience, let us first recall the definition of 
\emph{subanalytic} function, referring to 
\cite{Bierstone-Milman, Dries-Miller, Lojasiewicz1} for all details,
and to the recent  
\cite{BoDaLe} for the proof of the result that will be used in what follows.

\begin{definition}
\label{def-subanalitic}
\begin{enumerate}
\item A subset $A\subset \R^d$ is called \emph{semianalytic} if for every $x\in \R^d$ there exists a neighborhood $V$ such that
\begin{equation}
\label{semi-anal}
A\cap V = \cup_{i=1}^p \cap_{j=1}^q \{ x \in V \, : \ f_{ij}(x) =0, \ g_{ij}(x) >0 \},
\end{equation}
where for every $1\leq i \leq p$ and $1\leq j \leq q$ the functions $f_{ij},\, g_{ij} : V \to \R$ are analytic.
\item
We call  a set $A\subset \R^d$ \emph{subanalytic} if every $x\in \R^d$ admits  a neighborhood $V$ such that
there exists $B \subset \R^d \times \R^m$, for some $m\geq 1$, with
\begin{equation}
\label{sub-anal}
A\cap V = \pi_1 (B)  \text{ and }  B \text{ is a bounded semianalytic subset of $\R^d \times \R^m$},
\end{equation}
$\pi_1:  \R^d \times \R^m \to \R^d$ denoting the projection on the first component.
\item We say that a function $E: \R^d \to (-\infty,+\infty]$ is \emph{subanalytic}  if its graph is a subanalytic subset of $\R^d \times \R$.
\end{enumerate}
\end{definition}
As the above definition shows, the concept of \emph{subanalytic function} has a clear geometric character.
Without entering into details, let us recall that 
a fundamental example of subanalytic sets (hence of subanalytic functions)
is provided by \emph{semialgebraic sets}, i.e.\ sets $A \subset \R^d$ of the form
\begin{equation}
\label{semialgebraic}
A= \cup_{i=1}^p \cap_{j=1}^q \{ x \in V \, : \ f_{ij}(x) =0, \ g_{ij}(x) >0 \} \qquad \text{with }  f_{ij},\, g_{ij} : \R^d \to \R \text{ \emph{polynomial} functions} \end{equation}
for all $1\leq i \leq p$ and $1\leq j \leq q$.

We now consider for the functional $\calE$ 
the condition
\begin{equation}
\label{basic-hyp-last-section}
\text{for every\ \ $t\in [0,T]$\ \ the functional\ \ 
$u \mapsto \ene tu$\ \ is subanalytic}.
\end{equation}

To fix ideas, 
we may think of the case in which 
$\ene tu = E(u) - \langle \ell(t), u\rangle$, with $\ell \in \rmC^1([0,T];\xfin)$
and $E : \xfin\to \R$ of class $\rmC^1$ and subanalytic.
Thanks to \cite[Thm.\ 3.1]{BoDaLe}, 
for every $t\in [0,T]$, $\ene t{\cdot}$ complies with the 
{\L}ojasiewicz inequality \eqref{true-Loja}.
All in all, also in view of this result,
we can state the following theorem. 

\begin{theorem}
In the setting of \eqref{Ezero}--\eqref{P_t},
assume in addition the subanalyticity 
\eqref{basic-hyp-last-section},
and that $\cE\in \rmC^3([0,T]\times\xfin)$ fulfills the transversality conditions.
Consider a sequence 
$(u_{\eps_n}^0)_n$ of initial data for  \eqref{e:sing-perturb} such that 
\begin{equation*}
u_{\eps_n}^0 \to u_0 \quad \text{as } n\to\infty.
\end{equation*}
Then there exist a (not relabeled) subsequence and a curve 
$u \in \mathrm{B}([0,T];\xfin)$ such that 
convergences \eqref{3converg_var}--\eqref{4converg_var} hold,
$u(0)=u_0$, and $u$ is a Balanced Viscosity solution to \eqref{e:lim-eq}.
\end{theorem}

This result is a consequence of the fact that,
thanks to Proposition \ref{prop:7.3},
all the hypotheses of Theorem \ref{mainth:1} are in force,
and therefore the statement holds true
with $u$ being a Dissipative Viscosity solution to 
\eqref{e:lim-eq}.
Moreover, due to  the {\L}ojasiewicz inequality \eqref{true-Loja}, 
which is implied by \eqref{basic-hyp-last-section},
the Dissipative Viscosity solution $u$ improves
to a Balanced Viscosity solution 
in view of Theorem \ref{mainth:2}  
(cf.\ also Remark \ref{rmk:Loja}).


\end{document}